\documentclass[12pt,centertags,oneside]{amsart}
\RequirePackage[utf8]{inputenc}
\usepackage{amstext,amsthm,amscd,typearea,hyperref}
\usepackage{amsmath}
\usepackage{amssymb}
\usepackage{a4wide}
\usepackage[mathscr]{eucal}
\usepackage{mathrsfs}
\usepackage{typearea}
\usepackage{charter}
\usepackage{pdfsync}
\usepackage[a4paper,width=16.2cm,top=3cm,bottom=3cm]{geometry}

\usepackage{multicol}

\numberwithin{equation}{section}

\newtheorem{theorem}{Theorem}[section]
\newtheorem{definition}[theorem]{Definition}
\newtheorem{proposition}[theorem]{Proposition}
\newtheorem{corollary}[theorem]{Corollary}
\newtheorem{lemma}[theorem]{Lemma}
\newtheorem{remark}[theorem]{Remark}

\newcommand{\PSL}{\rm PSL}

\newcommand{\SU}{{\rm SU}}
\newcommand{\PSU}{{\rm PSU}}

\newcommand{\SL}{{\rm SL}}

\newcommand{\Aut}{{\rm Aut}}

\newcommand{\supp}{{\rm supp}}

\newcommand{\diff}{{\rm d}}

\newcommand{\omegaFS}{\omega_{\text{FS}}}

\newcommand{\del}{\partial}

\newcommand{\dist}{\mathop{\mathrm{dist}}\nolimits}

\newcommand{\ddc}{{\rm dd^c}}

\newcommand{\id}{{\rm id}}

\newcommand{\B}{\mathbb{B}}
\newcommand{\DD}{\mathbb{D}}

\newcommand{\C}{\mathbb{C}}

\newcommand{\N}{\mathbb{N}}

\newcommand{\R}{\mathbb{R}}

\renewcommand\P{\mathbb{P}}

\newcommand{\E}{\mathbf{E}}

\newcommand{\lp}{\langle}
\newcommand{\rp}{\rangle}
\newcommand{\Fix}{\mathrm{Fix}}
\newcommand{\Tr}{\mathrm{Tr}}

\newcommand{\Lip}{{\rm Lip}}
\newcommand{\area}{{\rm area}}

\title{Random products of matrices: a dynamical point of view}

\author{Tien-Cuong Dinh, Lucas Kaufmann and Hao Wu}
\address{Department of Mathematics,  National University of Singapore - 10, Lower Kent Ridge Road - Singapore 119076}
\email{matdtc@nus.edu.sg; lucaskaufmann@nus.edu.sg; e0011551@u.nus.edu}

\date{}

\begin{document}

\begin{abstract}
We study random products of matrices in $\SL_2(\C)$ from the point of view of holomorphic dynamics. For non-elementary measures with finite first moment we obtain the exponential convergence towards the stationary measure in Sobolev norm. As a consequence we obtain the exponentially fast equidistribution of forward images of points towards the stationary measure. We also give a new proof of the Central Limit Theorem for the norm cocycle under a second moment condition, originally due to Benoist-Quint, and obtain some general regularity results for stationary measures.
\end{abstract}

\maketitle

\tableofcontents

%

\section{Introduction and main results}

Let $G$ be the group $\SL_2(\C)$ of complex $2 \times 2$ matrices with determinant one and let $\mu$ be a probability measure on $G$. It is a classical problem to study random products of the form $g_n \cdots g_1$ where the $g_i$ are independent and identically distributed (i.i.d.) matrices with law $\mu$. This is a very rich theory with many beautiful results. A standard reference is the book \cite{bougerol-lacroix}. For a more recent account that deals with more general Lie group actions, the reader may also consult \cite{benoist-quint:book}. The goal of this paper is to revisit this problem using  the point of view of holomorphic dynamics. This is inspired by our recent work \cite{DKW:correspondences}. We hope that our methods can be applied in higher dimensions and can give a simplified treatment of known results.
\vskip5pt
The group $G$ acts naturally on the complex projective line $\P^1$. In the standard affine coordinate of $\P^1 = \C \cup \{\infty\}$ a matrix $\left( \begin{smallmatrix} a & b \\ c & d \end{smallmatrix} \right)$ acts  via the M\"obius transformation $z \mapsto \frac{az+b}{cz+d}$. This allows us to identify the group $\Aut(\P^1)$  of holomorphic automorphisms of  $\P^1$ with the group $ \PSL_2(\C)= \SL_2(\C) \slash \{\pm \text{Id} \}$. In what follows we will also denote this group by $G$ and we keep denoting by $\mu$ the probability measure induced on $\PSL_2(\C)$ by the measure $\mu$ on $\SL_2(\C)$. This shouldn't cause any confusion.

The probability measure $\mu$ defines a positive closed $(1,1)$-current on $\P^1 \times \P^1$ given by  $$[\Gamma_\mu]:= \int_G [\Gamma_g]\, \diff \mu(g),$$
where $[\Gamma_g]$ is the current of integration along the graph $\Gamma_g$ of $g \in \Aut(\P^1)$. The reader may consult \cite{demailly:agbook} and \cite{dinh-sibony:cime} for background material on currents on complex manifolds.

The current $[\Gamma_\mu]$ can be seen as the graph of a generalized correspondence, which we will denote by $f_\mu$. When the support of $\mu$ is finite $f_\mu$ is a correspondence in the usual sense, that is, $[\Gamma_\mu]$ is an effective one-dimensional cycle on $\P^1 \times \P^1$. In this case $f_\mu$ can be seen as a multivalued holomorphic map.

 This generalized correspondence acts on a current $T$ on $\P^1$ (e.g.\ a continuous function, a positive measure or a differential form) by the formula $$f^*_\mu (T) := \int_G g^*T \, \diff \mu(g),$$ or equivalently $f^*_\mu (T) = (\pi_1)_*(\pi_2^*(T) \wedge [\Gamma_\mu])$. We can also define $(f_\mu)_*$ by interchanging the roles of $\pi_1$ and $\pi_2$ or, equivalently, by replacing $g^*$ by $g_*$ in the above formula.

For a continuous function $\varphi$ on $\P^1$ we get
\begin{equation} \label{eq:transition-operator}
 f^*_\mu (\varphi)(x) = \int_G \varphi(g \cdot x) \, \diff \mu(g), \quad x \in \P^1,
\end{equation}
which is the standard Markov-Feller operator (or transfer operator) associated with $\mu$. Dually, if $m$ is a probability measure on $\P^1$ we have $(f_\mu)_* m = \mu \ast m$, the convolution of $\mu$ and $m$ (see \cite{bougerol-lacroix} for more details). A probability measure $m$ on $\P^1$ is called \textbf{stationary} with respect to $\mu$ if $\mu \ast m = m$, or equivalently if $m$ is $(f_\mu)_*$-invariant.

For $n \geq 1$ we define $f^n_\mu$ to be the correspondence associated with the convolution measure $\mu^{*n} = \mu * \cdots * \mu$ ($n$ times) which is the pushforward of the product measure $\mu^{\otimes n}$ on $G^n$ by the map $(g_1, \ldots, g_n) \mapsto g_n \cdots g_1$. When $\mu$ is finitely supported we recover the usual notion of iteration of a correspondence.

We say that $\mu$ is \textbf{non-elementary} if its support does not preserve a finite subset of $\P^1$ and if the semi-group generated by $\supp (\mu)$ is not relatively compact. It is a result of Furstenberg that a non-elementary measure admits a unique stationary measure (see \cite[II.4.1]{bougerol-lacroix} and Remark \ref{remark:unique-measure}).

Our first main result is the following. See Theorem \ref{thm:exponential-convergence} and also Remark \ref{remark:unique-measure} for the precise statement. See also Definition \ref{def:moments} for more on moment conditions on $\mu$.

\begin{theorem} \label{thm:exponential-convergence-main}
Let $\mu$ be a non-elementary probability measure on $G=\PSL_2(\C)$ and let $\nu$ be its unique stationary measure. Assume that $\mu$ has  a finite first moment, i.e.\ $\int_G \log \|g\| \diff \mu (g) < + \infty$.  Then the iterates of the transfer operator associated with $\mu$ converge exponentially fast to $\nu$ with respect to the Sobolev norm $\|\cdot\|_{W^{1,2}}$ on test functions. 
\end{theorem}

When the measure $\mu$ has a finite exponential moment, that is, when $\int_G \|g\|^\alpha \diff \mu (g) < + \infty$ for some $\alpha > 0$, the exponential convergence of the transfer operator towards the stationary measure is a fundamental result of Le Page \cite{lepage:theoremes-limites}. The convergence in this case is for test functions in some H\" older space and it has many important consequences such as the Central Limit Theorem mentioned below, the Large Deviation Theorem and other analogues of classical limit theorem for i.i.d.\ random variables.
\vskip5pt
The results that follow will be consequences of Theorem \ref{thm:exponential-convergence-main}. The first one says that the forward images of any given point  $a \in \P^1$ by the generalized correspondence $f_\mu$ converge to $\nu$ exponentially fast and uniformly in $a$.

\begin{theorem} \label{thm:equidistribution-images}
Let $\mu$ be a non-elementary probability measure on $G=\PSL_2(\C)$. Assume that $\int_G (\log \|g\|)^{1 + \epsilon} \, \diff \mu(g) < + \infty$  for some $\epsilon > 0$. Then, there is a constant $0< \gamma < 1$ such that for any $a \in \P^1$ and every test function $\varphi$ of class $\mathcal C^\beta$ on $\P^1$, with $0<\beta \leq 1$,  we have
\begin{equation}
\big| \big \langle (f_\mu^n)_*\delta_a - \nu, \varphi \big \rangle \big | \leq A_\beta \|\varphi\|_{\mathcal C^\beta} \gamma^{\beta n} \quad \text{for every} \quad n \geq 0,
\end{equation}
where $A_\beta > 0$ is a constant independent of $n$, $a$ and $\varphi$.
\end{theorem}

Next, we give a new proof of the follwing known Central Limit Theorem for the random variables $\log \frac{\|g_n \cdots g_1 \cdot v\|}{\|v\|}$ where $v$ is any non-zero vector in $\C^2$. The Lyapunov exponent is defined in Section \ref{sec:CLT}.

\begin{theorem}[Central Limit Theorem] \label{thm:CLT}
Let $\mu$ be a probability measure on $G=\PSL_2(\C)$. Assume that $\mu$ is non elementary and has a finite second moment, i.e.\ $\int_G (\log \|g\|)^2 \, \diff \mu(g) < + \infty$. Let $\gamma$ be the Lyapunov exponent of $\mu$. Then there exists a number $\sigma >0$ such that for any $v \in \C^2 \setminus\{0\}$. 
\begin{equation} \label{eq:CLT}
\frac{1}{\sqrt n} \bigg( \log \frac{\|g_n \cdots g_1 \cdot v\|}{\|v\|} - n \gamma \bigg) \longrightarrow \mathcal N(0;\sigma^2) \quad \text{ in law},
\end{equation}
where $N(0;\sigma^2)$ is the centred normal distribution with variance $\sigma^2$.
\end{theorem}

Under an exponential moment condition, the above result is mainly due to Le Page  \cite{lepage:theoremes-limites} and was later refined by other authors (see  for instance \cite{guivarch-raugi,golsheid-margulis}).  The question of whether this condition could be relaxed to an (optimal) second moment condition remained open until very recently, when Benoist-Quint gave an affirmative answer, \cite{benoist-quint:CLT}. Our proof is independent of theirs and, in particular, does not rely on an a priori knowledge of the regularity of $\nu$ (although we also obtain such results later in Section \ref{sec:regularity}). We expect that our method can be generalized to cover the general case of Benoist-Quint.
\vskip5pt
Our final result concerns the regularity of stationary measures. If  $\mu$ is a non-elementary probability measure and $\nu$ is the associated stationary measure, the regularity of $\nu$ will depend on moment conditions on $\mu$. The statement of our main result (Theorem \ref{t:general-sp} below) and its proof rely on the theory of superpotentials introduced by Sibony and the first author \cite{dinh-sibony:acta}. We state here some more concrete consequences (see Corollaries \ref{cor:reg-expmoment} and \ref{cor:reg-firstmoment}) and refer to Section \ref{sec:regularity} for the general statements. It is worth mentioning that similar regularity results can be found in the literature (see Remarks \ref{rmk:holder-regular} and \ref{rmk:BQregularity}), although they are obtained by completely different methods. Here $\DD(a,r)$ denotes the disc of radius $r$ and center $a$ with respect to the standard metric  on $\P^1$. 

\begin{theorem}
Let $\mu$ be a non-elementary probability measure on $G=\PSL_2(\C)$. 
\begin{enumerate}
\item If $\mu$ has a finite exponential moment, i.e.\ $\int_G \|g\|^p \, \diff \mu(g) < +\infty$ for some $p>0$, then there are constants $c, \alpha > 0$ such that  $\nu(\DD(a,r)) \leq c r^\alpha$ for every $a \in \P^1$ and $0 < r \leq 1$.
\item If $\mu$ has a finite first moment, i.e. $\int_G \log \|g\| \, \diff \mu(g) < +\infty $ then there are constants $c, \alpha > 0$  such that $\nu(\DD(a,r))\leq c|\log r|^{-\alpha}$ for every $a \in \P^1$ and $0 < r \leq 1$.
\end{enumerate}
\end{theorem}

\noindent  \textbf{Acknowledgements:} This paper was partially written during the visit of the first author to the Institute of Mathematical Sciences and Department of Mathematics in Chinese University of Hong Kong. He would like to thank these organisations and Prof. Conan Leung for their very warm hospitality. This work was supported by the NUS grants C-146-000-047-001,  AcRF Tier 1 R-146-000-248-114  and R-146-000-259-114.

\section{Action on Sobolev space and convergence to the stationary measure}

This section  is devoted to the proof of Theorem \ref{thm:exponential-convergence-main}. We show that when $\mu$ has a finite first moment, i.e. when $\int_G \log\|g\| \, \diff \mu(g) <+ \infty$, the operator $f^*_\mu$ acts continuously on the Sobolev space $W^{1,2}$. Later on, we prove that when $\mu$ is non-elementary this action has a spectral gap. As a consequence, we get an exponentially fast convergence of the transfer operator towards the stationary measure. 

Consider the space $$L^2_{(1,0)} := \big\{\phi : \phi \text{ is  a } (1,0) \text{-form on } \P^1 \text{ with } L^2 \text{ coefficients} \big\}$$
equipped with the norm
\begin{equation} \label{eq:L2norm}
\|\phi\|_{L^2} := \Big( \int_{\P^1} i \phi \wedge \overline \phi \Big)^{1/2}.
\end{equation}

The space $L^2_{(0,1)}$ and the corresponding norm are defined analogously.

\begin{proposition} \label{prop:CS-ineq}
Let $\mu$ be a probability measure on $G= \PSL_2(\C)$ and let $f_\mu$ be the associated generalized correspondence. Then the operator $f_\mu^*$, which is well-defined on smooth $(1,0)$-forms, extends to a  bounded linear operator $f_\mu^*:L^2_{(1,0)}\to L^2_{(1,0)}$ with norm bounded by $1$. In other words, for $\phi$ in $L^2_{(1,0)}$ we have the inequality $\|f_\mu^*\phi\|_{L^2}\leq \|\phi\|_{L^2}$. Moreover, the equality holds 
if and only if $g_1^*\phi = g_2^*\phi$ for $\mu \otimes \mu$ almost every $(g_1,g_2)$.
\end{proposition}

\begin{proof}
By a direct computation we have
\begin{align} \label{eq:CS}
f_\mu^*(i \phi \wedge \overline \phi) - i f^*_\mu \phi \wedge \overline{f^*_\mu \phi} &= \int_G g^*(i \phi \wedge \overline \phi) \, \diff \mu (g) - i \Big( \int_G g^* \phi \, \diff \mu (g) \Big) \wedge \overline{\Big( \int_G g^* \phi \, \diff \mu (g) \Big)} \nonumber \\ &= \frac{1}{2} \int_G \int_G i(g_1^*\phi - g_2^* \phi) \wedge \overline{(g_1^*\phi - g_2^* \phi)} \, \diff \mu(g_1) \diff \mu(g_2).
\end{align}

These identities are clear for smooth $\phi$. We obtain the general case by the density of smooth forms in $L^2_{(1,0)}$.

Notice that the right hand side of (\ref{eq:CS}) is a positive measure on $\P^1$. By integrating the left hand side of (\ref{eq:CS}) over $\P^1$ and using the fact that the action of $f_\mu^*$ on measures preserves the total mass we get $\|\phi\|^2_{L^2} - \|f_\mu^* \phi\|^2_{L^2} \geq 0$. This is the desired inequality. From (\ref{eq:CS}) it is also clear that  $\|f^*\phi\|_{L^2} = \|\phi\|_{L^2}$ if and only if  $g_1^*\phi = g_2^*\phi$ holds for all $(g_1,g_2)$ outside a set of zero measure for $\mu \otimes \mu$.
\end{proof}

Consider now the Sobolev space $W^{1,2}$ of real valued $L^1$ functions on $\P^ 1$ with finite  $\|\cdot\|_{W^{1,2}}$ norm, where
$$\|h\|_{W^{1,2}} := \Big| \int_{\P^1}  h \, \omegaFS \Big| + \|\partial h\|_{L^2}$$ 
and $\omegaFS$ stands for the Fubini-Study form on $\P^1$.

The following proposition was proved in \cite{DKW:correspondences}.

\begin{proposition}\label{prop:equinorm}
		Let $U$ be a non-empty open subset of $\P^1$. Then the following norms on $W^{1,2}$ are equivalent to the norm $\|\cdot\|_{W^{1,2}}$.
\vspace{-15pt}
\begin{multicols}{2}
\begin{enumerate}
        \item $\|h\|_1 := \|h\|_{L^1} +\|\partial h\|_{L^2}$ \medskip
        \item $ \|h\|_2 := \|h\|_{L^2}+\|\partial h\|_{L^2}$
        \columnbreak
        \vspace*{\fill}
        \item $\|h\|_3 := |\int_U h \, \omegaFS|+\|\partial h\|_{L^2}$ \medskip
        \item $\|h\|_4 := \int_U |h| \omegaFS+\|\partial h\|_{L^2}$.
\end{enumerate}
\end{multicols}
\end{proposition}

Here and in what follows $\|g\| := \sup_{v \in \C^2 \setminus \{0\}} \frac{\|g \cdot v\|}{ \|v\|}$ will denote the operator norm of the matrix $g$. Notice that, for $g \in \SL_2(\C)$, we have $\|g\| \geq 1$ and $\|g^{-1}\| = \|g\|$. This follows from Cartan's decomposition (see the proof of Lemma \ref{lemma:jacobian-estimate} below).

\begin{proposition} \label{prop:f^*-bounded}
Let $f_\mu$ be the generalized correspondence associated with $\mu$ on $G=\PSL_2(\C)$. Assume that $\int_G \log\|g\|\, \diff \mu(g) < + \infty$. Then the transfer operator $f_\mu^*$ acting on smooth functions extends to a bounded linear operator from $W^{1,2}$ to itself.
\end{proposition}

For the proof we need some preliminary results.

\begin{lemma} \label{lemma:jacobian-estimate}
We have  $g^* \omegaFS \leq \|g\|^4 \cdot \omegaFS$ and $g_* \omegaFS \leq \|g\|^4 \cdot \omegaFS$  for every $g \in \PSL_2(\C)$. 
\end{lemma}

\begin{proof}
Using Cartan's decomposition we can write any element $g$ in $\PSL_2(\C)$ as $g = k a  k'$, where $k,k' \in \SU(2)$ and $a= \left(\begin{smallmatrix} \lambda & 0 \\ 0 & \lambda^{-1} \end{smallmatrix}\right)$, for some $\lambda \geq 1$. We see that $\|g\| = \lambda$. Since $\SU(2)$ preserves $\omegaFS$ and $\|g\| = \|a\|$ we can assume that $g$ is of the form $\left(\begin{smallmatrix} \lambda & 0 \\ 0 & \lambda^{-1} \end{smallmatrix}\right)$, for some $\lambda \geq 1$.

In the standard affine coordinate of $\P^1$ we have $\omegaFS = \frac{i}{2\pi} \frac{\diff z \wedge \diff \bar z}{(1+|z|^2)^2}$ and $g(z) = \lambda^2 z$, so $$g^* \omegaFS = \lambda^4 \frac{i}{2\pi} \frac{\diff z \wedge \diff \bar z}{(1+ \lambda^4|z|^2)^2} \leq  \lambda^4 \frac{i}{2\pi} \frac{\diff z \wedge \diff \bar z}{(1+ |z|^2)^2} = \lambda^4 \omegaFS = \|g\|^4 \omegaFS,$$
which proves the first inequality.

For the second inequality we apply the above argument for $g^{-1}$ instead of $g$ and use that $\|g^{-1}\| = \|g\|$.
\end{proof}

The following exponential estimate will be crucial for us. It will also be used in the proof of Theorem \ref{thm:equidistribution-images} in Section \ref{sec:equidistribution} and will be important in Section \ref{sec:regularity}.

\begin{proposition}[Moser-Trudinger estimate \cite{moser:trudinger}] \label{prop:exponential-estimate}
Let $\mathcal F$ be a bounded family in $W^{1,2}$. Then there are constants $A>0$ and $\alpha > 0$, depending on ${\mathcal F}$, such that 
$$\int_{\P^1} e^{\alpha \varphi^2} \omegaFS  \leq A \quad \text{for every} \quad \varphi \in \mathcal F.$$
\end{proposition}

\begin{proof}[Proof of Proposition \ref{prop:f^*-bounded}]
We need to show that $\|f_\mu ^* \varphi \|_{W^{1,2}}$ is uniformly bounded in $\varphi$ if $\|\varphi\|_{W^{1,2}} \leq 1$. For such $\varphi$ we have, from Proposition \ref{prop:CS-ineq}, that $\|\del f^*_\mu (\varphi)\|_{L^2} \leq 1$, so using Proposition \ref{prop:equinorm} it remains to check that $\|f_\mu ^* \varphi \|_{L^2}$ is uniformly bounded.

Let $\alpha$ and $A$ be as in Proposition \ref{prop:exponential-estimate} for $\mathcal F= \{\varphi: \|\varphi\|_{W^{1,2}} \leq 1\}$.
From Jensen's inequality and  Lemma \ref{lemma:jacobian-estimate}  we have $$\exp\left( \int_{\P^1} \alpha (g^*\varphi)^2 \omegaFS \right)\leq \int_{\P^1} e^{\alpha (g^* \varphi)^2} \omegaFS = \int_{\P^1} e^{\alpha \varphi^2} g_*\omegaFS \leq \|g\|^4 \int_{\P^1} e^{\alpha \varphi^2} \omegaFS \leq A \|g\|^4.$$ Taking the logarithm gives
\begin{equation} \label{eq:g*phi-L2norm}
\|g^* \varphi\|_{L^2} \leq A_2 + A_3 \log\|g\|,
\end{equation}
for some constants $A_2,A_3 >0$.

By Cauchy-Schwarz we have $$\|f^*_\mu \varphi\|_{L^2} = \Big \|\int_G g^* \varphi \, \diff \mu(g) \Big\|_{L^2} \leq \int_G \|g^* \varphi\|_{L^2} \, \diff \mu(g).$$ Since $\int_G \log\|g\| \, \diff \mu(g)$ is finite by assumption, it follows that $\|f^*_\mu \varphi\|_{L^2} \leq A_4$ for every $\varphi \in \mathcal F$ for some constant $A_4 > 0$. This finishes the proof.
\end{proof}

\subsection{Non-elementary measures}

Let $\mu$ be a probability measure on $\PSL_2(\C)$. We will denote by $T_\mu$ the smallest closed sub-semigroup of  $\PSL_2(\C)$ containing the support of $\mu$.

\begin{definition} \label{def:elementary}
Let $R$ be a subset of $\PSL_2(\C)$. We say that $R$ is \textbf{elementary} if either $R$ is conjugated to a subset of $\PSU(2)$ or if there is a finite subset of $\P^1$ which is invariant by every element of $R$.
We say that a probability measure $\mu$ on $\PSL_2(\C)$ is elementary if $\supp (\mu)$ is an elementary set.
\end{definition}

\begin{remark} \label{rmk:elementary}
(i) It is easy to see that $R$ is elementary if and only if the closed semigroup generated by $R$ is elementary. In particular $\mu$ is elementary if and only if $T_\mu$ is elementary.

(ii) A subset $R$ of $\PSL_2(\C)$ is conjugated to a subset of $\PSU(2)$ if and only if the group generated by $R$ is relatively compact. This follows from the fact that if the semigroup generated by $R$ is relatively compact then there exists an $R$-invariant inner product on $\C^2$, obtained by averaging the standard inner product.

(iii) We can view $\P^1$ as the boundary of the $3$-dimensional hyperbolic space $\mathbb H ^3$. Then, any M\"obius transformation of $\P^1$ extends to a homeomorphism of $\overline{\mathbb H^3} = \mathbb H^3 \cup \P^1$, called the Poincar\'e extension, that preserves the standard hyperbolic metric on $\mathbb H ^3$. In this context, a subset $R$ of $\PSL_2(\C)$ is elementary if and only if it admits a finite orbit in $\overline{\mathbb H^3}$, see \cite{beardon}.
\end{remark}

The following result is probably well-known. We include a proof for the convenience of the reader. Recall that an element $g$ of $\Aut(\P^1)$ different from the identity is conjugated to either $z \mapsto z + 1$ or $z \mapsto \lambda z$ for some $\lambda \in \C \setminus \{0,1\}$. In  the former case, $g$ is called \textit{parabolic} and in the latter, $g$ is called \textit{elliptic} if $|\lambda| = 1$ or \textit{loxodromic} if $|\lambda| \neq 1$, see also the appendix below.

\begin{lemma} \label{lemma:non-elementary}
Let $n \geq 1$. Then $\mu$ is non-elementary if and only if $\mu^{*n}$ is non elementary.
\end{lemma}

\begin{proof}
Notice that $T_{\mu^{*n}} \subset T_\mu$, so if $\mu$ elementary then so is $\mu^{*n}$.

Suppose now that $\mu$ is non-elementary and fix $n$. It follows from Lemma \ref{Rn lox} in the Appendix that $T_\mu$ contains a loxodromic element $g_0$.  Then $T_{\mu^{*n}}$ contains a loxodromic element, namely, $g_0^n$. In particular  $T_{\mu^{*n}}$ is non-compact and cannot be conjugated to subset of $\PSU(2)$.

To finish the proof we need to show that $\supp (\mu^{*n})$ leaves no finite set invariant. Suppose $F \subset \P^1$ is finite and invariant by $\supp (\mu^{*n})$. Then $F$ is also invariant by $T_{\mu^{*n}}$. Since $\id \neq g^n_0 \in T_{\mu^{*n}}$ is loxodromic, this implies that $F \subset \Fix(g_0)$. As $\mu$ is non-elementary, we can find another loxodromic element  $g_1 \in T_\mu$ whose fix point set is disjoint from $\Fix(g_0)$ (see \cite[Thm.\ 5.1.3]{beardon}).  Repeating the preceding argument for $g_1$ gives $F \subset \Fix(g_1)$. This implies that $F = \varnothing$, completing the proof. 
\end{proof}

The main result of this section is the following. 

\begin{proposition}\label{prop:normlessthan1}
Let $\mu$ be a non-elementary probability measure on $\PSL_2(\C)$. Then there exists an $N \geq 1$ such that the norm of the operator $(f_\mu^N)^*:L^2_{(1,0)} \to L^2_{(1,0)}$ is strictly less than $1$.
\end{proposition}

\begin{proof}
For $n\geq 1$ introduce  $$R^n :=\{g_n \cdots g_2 g_1 : g_i\in \supp(\mu)\} \quad \text{ and } \quad S^n:=\{gh^{-1} : g,h \in R^n\}.$$ Notice that $R^n$ is a dense subset of the support of the $\mu^{*n}$.

	By Proposition \ref{prop:CS-ineq}, $\|(f^n_{\mu})^*\|\leq 1$ for every $n \geq 1$. Suppose by contradiction that  $\|(f^n_{\mu})^*\| = 1$ for every $n \geq 1$ . We will show that in this case $\mu$ must be elementary.
	
	Since $\|f_{\mu}^*\|= 1$,  there exists a sequence of  $(1,0)$-forms $\{\phi_n\}_{n\geq 0}$ such that $\|\phi_n\|_{L^2}=1$ and $\|f_\mu^*(\phi_n)\|_{L^2}\to 1$. By compactness, the sequence $\{i\phi_n\wedge\overline{\phi_n}\}_{n\geq 0}$ of probability measures admits a subsequence, which we still denote by $\{i\phi_n\wedge\overline{\phi_n}\}_{n\geq 0}$ for simplicity, that converges to a probability measure $m$.
	
	By the proof of Proposition \ref{prop:CS-ineq}, the measures $$\nu_n :=\int_{G \times G} i(g_1^*\phi_n - g_2^* \phi_n) \wedge \overline{(g_1^*\phi_n - g_2^* \phi_n)} \, \diff \mu(g_1) \diff \mu(g_2)$$ tend to zero as $n \to \infty$. In particular, 
		$$\|\nu_n\| = \int_{G \times G} \|g_1^*\phi_n - g_2^* \phi_n\|^2_{L^2}  \, \diff \mu(g_1) \diff \mu(g_2) \longrightarrow 0$$
as $n \to \infty$.
		
By Cauchy-Schwarz and the fact that $\|g^* \phi_n\|_{L^2} = \|\phi_n\|_{L^2}=1$ for $g \in \PSL_2(\C)$  we have
\begin{align*}
\|g_1^*(i\phi_n\wedge\overline{\phi_n}) - &g_2^*(i\phi_n\wedge\overline{\phi_n})\|_{L^1} =\|ig_1^*\phi_n \wedge \overline{(g_1^* \phi_n - g_2^* \phi_n)} + i(g_1^* \phi_n - g_2^* \phi_n) \wedge \overline{g_2^* \phi_n} \|_{L^1} \\
&\leq \|g_1^* \phi_n \|_{L^2} \|g_1^*\phi_n - g_2^* \phi_n \|_{L^2}+\|g_2^*\phi_n\|_{L^2}\|g_1^*\phi_n - g_2^*\phi_n\|_{L^2} \\
&= 2  \|g_1^*\phi_n - g_2^* \phi_n \|_{L^2},
\end{align*}
	 so	$$\int_{G \times G}  \|g_1^*(i\phi_n\wedge\overline{\phi_n})-g_2^*(i\phi_n\wedge\overline{\phi_n})\|_{L^1}^2 \, \diff \mu(g_1) \diff \mu(g_2) \longrightarrow 0$$
	 as $n \to \infty$.
	
By Lebesgue's dominated convergence theorem, it follows that $$\int_{G \times G}  \|g_1^*(m)-g_2^*(m)\|^2  \, \diff \mu(g_1) \diff \mu(g_2) = 0,$$
which implies that $g_1^*(m)=g_2^*(m)$  for $\mu \otimes \mu$ almost every $(g_1,g_2)$.

\vskip5pt
	
\noindent	\textbf{Claim:} $g_1^*(m)=g_2^*(m)$ for all $g_1,g_2\in\supp(\mu)$.

\vskip5pt

\noindent	Indeed, we know that $g_1^*(m)=g_2^*(m)$ holds for $g_1$ and $g_2$ on a set of full $\mu$-measure. Now, such a set is dense in the support of $\mu$ for the standard distance on $\PSL_2(\C)$ and $g \mapsto g^* m$ is continuous with respect to this distance. Hence $g_1^*(m)=g_2^*(m)$ for all $g_1,g_2\in\supp(\mu)$ and the claim is proved.
\vskip3pt
The claim is equivalent to
\begin{equation*}
(g_1 g_2^{-1})^* m = m \quad \text{for every} \quad g_1,g_2 \in \supp (\mu). 
\end{equation*}
Which means $m$ is invariant by $S: = S^1$.

Since we also have $\|(f_{\mu}^n)^*\|=1$ by assumption,  we can replace $f_{\mu}$ by $f_{\mu}^n$ in the above proof and get, for each $n \geq 1$, a probability measure $m_n$ invariant by $S^n$.
\vskip3pt
We can now finish the proof. After replacing $f$ by $f^{N_2}$, $R$ by $R^{N_2}$ and $S$ by $S^{N_2}$ for some $N_2$ we may assume that $S$ contains a non-elliptic element $g_0$ different from the identity. This is possible by Lemma \ref{Rn lox} from the appendix. By the above discussion, there exists a probability measure $m_1$ invariant by the pullback by every element of $S$. In particular $g_0^* m_1=m_1$ and by iteration $(g_0^n)^*m_1 = m_1$ for every $n \geq 1$. Making $n \to \infty$ implies that $m_1 = \alpha_1\delta_x+\beta_1\delta_y$, with  $\alpha_1$, $\beta_1 \geq 0$  and $\alpha_1 + \beta_1=1$, where $x$ and $y$ are the fix points of $g_0$ (if $g_0$ is parabolic we set $x=y$). Notice that $S \subset S^n$ for every $n \geq 1$ so the measure $m_n$ is also invariant by $g_0$. Hence $m_n=\alpha_n \delta_x+\beta_n \delta_y$ with $\alpha_n + \beta_n=1$.

We will show now that $\mu$ is elementary. Let $F^n$ be the largest finite $S^n$-invariant subset of $\P^1$. Notice that when $n \leq m$ we have $S^n \subset S^m$, hence $F^m \subset F^n$. We have $F^n \neq \varnothing$ for every $n \geq 1$, because $S^n$ preserves the atomic measure $m_n$. Also, since $F^1$ is invariant by $g_0$ we have $F^1 \subset \{x,y\}$. We separate in a few cases. 

\vskip3pt

\noindent \textsc{Case 1:} $F^1 = \{x\}$. In this case $S$ fixes $x$, which means that  there is a point $p$ such that $R$ maps $p$ to $x$. As $\varnothing \neq F^2 \subset F^1$ we have $F^2 = \{x\}$, so $S^2$ also fixes $x$. Hence there is a point $q$ such that $R^2$  maps $q$ to $x$. For $g \in R$ we have that $g^2 \in R^2$, so $g \cdot p = x = g^2 \cdot q $. Hence $g \cdot q = p$ for every $g \in R$. This implies that $p$ is $S$-invariant, so we must have $p=x$. We conclude that $g\cdot x =x $ for every $g \in R = \supp(\mu)$, so $\mu$ is elementary.

\vskip3pt

\noindent \textsc{Case 2:} $F^1 = \{x,y\}$ and $F^2 = \{x\}$ or $\{y\}$. In that case we can replace $f_\mu$ by $f_\mu^2$ and repeat the argument of Case 1, see also Lemma \ref{lemma:non-elementary}.

\vskip3pt

\noindent \textsc{Case 3:} $F^1 = \{x,y\}$ and $F^2 = \{x,y\}$. If $x=y$ we fall in Case 1, so we may assume $x \neq y$. In this case the set $\{x,y\}$ is $S$-invariant, which means that there are points $p,q$ such that $R$ maps $\{p,q\}$ to $\{x,y\}$ . Analogously, $\{x,y\}$ is $S^2$-invariant so there are points $r,s$ such that $R^2$ maps $\{r,s\}$ to $\{x,y\}$. For $g \in R$ we have that $g^2 \in R^2$, so $g \{p,q\} = \{x,y\} = g^2 \{r,s\} $. Hence $\{p,q\}= g\{r,s\}$ for every $g \in R$, which implies that $\{p,q\}$ is $S$-invariant. By the maximality of $F^1$ we get $\{x,y\} = \{p,q\}$. Hence  $R = \supp (\mu)$ maps $\{x,y\} $ to itself, so $\mu$ is elementary.

\vskip3pt

Summing up, we have shown that if $\|(f^n_{\mu})^*\| = 1$ for every $n \geq 1$ then $\mu$ must be elementary, thus completing the proof.
\end{proof}

Once we know that, up to taking iterates, $f_\mu^* : L^2_{(1,0)} \to L^2_{(1,0)} $ has norm less than one, we will have that $f_\mu^* : W^{1,2} \to W^{1,2}$ has a spectral gap. It is then well known how to use this to produce a stationary measure. This is the content of the next result.

\begin{theorem} \label{thm:exponential-convergence}
Let $\mu$ be a non-elementary probability measure on $G=\PSL_2(\C)$. Assume that $\mu$ has a finite first moment, i.e.\  $\int_G \log \|g\| \diff \mu (g) < + \infty$. Then $\mu$ admits a stationary measure $\nu$ that can be extended to a continuous linear functional on $W^{1,2}$ with the following properties 
\begin{enumerate}
\item There are constants $A >0$ and $0<\lambda<1$ such that 
\begin{equation*}
\Big \|(f_\mu^n)^*h- \langle \nu,h \rangle \Big \|_{W^{1,2}} \leq A \|\del h\|_{L^2} \lambda^n \quad \text{for every }  n \geq 0 \quad \text{and every } h \in W^{1,2}.
\end{equation*}
\item $|\langle \nu,h \rangle| \leq A'\|h\|_{W^{1,2}}$ for some constant $A'>0$ independent of $h$.
\end{enumerate}
In particular, $\nu$ has no mass on polar subsets of $\P^1$.
\end{theorem}

\begin{proof}
By Proposition \ref{prop:normlessthan1}  we may assume, after replacing $f_\mu$ by $f_\mu^N$ for some $N \geq 1$, that the norm of $f_\mu^*$ acting on $L^2_{(1,0)}$ is less than one. Let $0<\lambda < 1$ be its value.

For $h \in W^{1,2}$, let 
$$c_0:=\int_X h \omegaFS \qquad \text{and} \qquad h_0:=h-c_0$$ 
and define inductively  
$$c_n:=\int_X (f_\mu^*h_{n-1})\omegaFS \qquad \text{and} \qquad h_n:= f_\mu^*h_{n-1}-c_n.$$ 
Then
\begin{equation} \label{eq:push-forward-hn}
(f_\mu^n)^* h = h_n + c_n + c_{n-1}+ \cdots + c_1+ c_0.
\end{equation}

		By Proposition \ref{prop:f^*-bounded}, we have $h_n\in W^{1,2}$ for all $n$. Notice that $\lp \omegaFS, h_n \rp = 0$, so by  Poincar\' e-Sobolev inequality we have $\|h_n\|_{L^2} \leq A_1 \|\partial h_n\|_{L^2}$ for some constant $A_1>0$. We also have $\del h_n = f_\mu^* (\del h_{n-1})$ for every $n$. Then  $$\|h_n\|_{L^2}\leq A_1 \left \|\partial h_n \right\|_{L^2} = A_1 \|(f_\mu^n)^*(\partial h) \|_{L^2} \leq A_1\lambda^n\|\partial h\|_{L^2}.$$
	
	By Propositions \ref{prop:equinorm}  and \ref{prop:f^*-bounded}, there is a constant $A_2>0$ such that $\|f^* \varphi \|_{L^2}\leq A_2\|\varphi\|_{W^{1,2}}$ for every $\varphi \in W^{1,2}$. Hence, we have
	\begin{align} \label{eq:estimate-cn}
		 		|c_n|& =\Big |\int_X  (f_\mu^*h_{n-1})\omegaFS \Big  | \leq   \| f_\mu^*h_{n-1}  \|_{L^2} \leq   A_2\|h_{n-1}\|_{W^{1,2}} \nonumber \\
		& = A_2\|\partial h_{n-1}\|_{L^2} = A_2  \|(f_\mu^{n-1})^*(\partial h)  \|_{L^2} 
		\leq A_2 \lambda^{n-1}  \|\partial h\|_{L^2}. \nonumber
	\end{align}

Set  $c_h:= \sum^{\infty}_{k=0}c_k$ and define the linear functional $\nu$ by $$\langle \nu, h \rangle := c_h \quad \text{for } h \in W^{1,2}.$$ Clearly, this constant is finite and satisfies the estimate stated in (2) for a suitable constant $A'>0$. We also have $\langle \nu, \mathbf 1 \rangle = 1$ according to (\ref{eq:push-forward-hn}) and if $h$ is smooth and non-negative we have $\langle \nu, h \rangle = \lim_{n \to \infty} \int_{\P^1} (f^n_\mu)^*h\, \omegaFS \geq 0$. So, by Riesz Representation Theorem, $\nu$ defines a probability measure on $\P^1$. 

We have from (\ref{eq:push-forward-hn}) that
		\begin{align*}
		  \|(f_\mu^n)^*h - \langle \nu, h \rangle  \|_{L^2}&=   \|(f_\mu^n)^*h-c_h  \|_{L^2} = \Big  \|h_n-\sum_{k=n+1}^\infty c_k \Big  \|_{L^2} \leq \|h_n\|_{L^2}+\sum_{k=n+1}^\infty |c_k| \\
		&\leq A_1 \lambda^n \|\partial h\|_{L^2}+\sum_{k=n+1}^\infty A_2 \lambda^{k-1} \|\partial h\|_{L^2} 
		\leq A_3 \| \del h\|_{L^2} \lambda^n
		\end{align*}
	for some constant $A_3 > 0$.  On the other hand, by Proposition \ref{prop:equinorm} and the definition of $\lambda$, we obtain
\begin{align*} 
\|(f_\mu^n)^*h-c_h \|_{W^{1,2}}  & \lesssim \|(f_\mu^n)^*h-c_h \|_{L^2} + \|(f_\mu^n)^* (\del h) \|_{L^2} \\
& \leq  \|(f_\mu^n)^*h-c_h \|_{L^2} +  \lambda^n \|\partial h\|_{L^2} \\
& \leq A_4 \lambda^n \|\partial h\|_{L^2}
\end{align*}
for some constant $A_4>0$. 
Thus, we get (1)  for a suitable constant $A>0$.

In order to show that $\nu$  is stationary we notice that, from (1), we have $(f_\mu^n)^*h \to \langle \nu, h \rangle$ in $W^{1,2}$ for every $h \in W^{1,2}$. In particular, if $\varphi$ is a smooth test function, then$$ \langle \nu, \varphi \rangle =  \lim_{n \to \infty}(f_\mu^{n+1})^* \varphi = \lim_{n \to \infty} (f_\mu^n)^* f_\mu^* \varphi = \langle \nu, f_\mu^* \varphi \rangle = \langle (f_\mu)_* \nu,  \varphi\rangle,$$ showing that $(f_\mu)_*\nu = \nu$, that is, $\nu$ is stationary.

We now prove the last statement. If $E \subset \P^1$ is a polar set then, by definition, there is a quasi-subharmonic function $u$ on $\P^1$ such that $E\subseteq \{u=-\infty\}$. We may assume that $u \leq -1$ and $u$ is the limit of a decreasing sequence of smooth negative functions $u_n$ with $\ddc u_n\geq -\omegaFS$. Then $h := - \log (-u)$ belongs to $W^{1,2}$ and is the decreasing limit of the sequence $h_n := - \log (-u_n)$ which is bounded in $W^{1,2}$, see \cite{dinh-sibony:decay-correlations} and  \cite[Ex.1]{vigny:dirichlet}. The function $h$ is defined everywhere and is bounded from above, so $\lp \nu,h \rp$ coincides with the integral of $h$ with respect to $\mu$. The fact $h=-\infty$ on $E$ and that $\lp \nu,h \rp$ is finite imply that $\nu(E)=0$. The proof is now complete.
\end{proof}

\begin{remark} \label{remark:unique-measure}
As mentioned in the Introduction, it is well known since Furstenberg that a non-elementary measure admits a unique stationary measure. Hence, the measure $\nu$ in the above theorem is necessarily the unique $\mu$-stationary measure and our result says that the iterates of the transfer operator $f_\mu^*$ converge exponentially fast with respect to the Sobolev norm to the operator $\varphi \mapsto \lp \nu, \varphi \rp \mathbf 1$. The uniqueness of the stationary measure also follows from Theorem \ref{thm:equidistribution-images}.
\end{remark}

The proof of Theorem \ref{thm:exponential-convergence-main} follows immediately  from Theorem \ref{thm:exponential-convergence} and Remark \ref{remark:unique-measure}. The following consequence of Theorem \ref{thm:exponential-convergence} will be used later.

\begin{corollary} \label{cor:nu-norm}
Let $\mu$ and $\nu$ be as in Theorem \ref{thm:exponential-convergence}. Then
$\|\varphi\|_\nu:=|\langle \nu,\varphi\rangle|+ \|\partial\varphi\|_{L^2}$
defines a norm on $W^{1,2}$ which is equivalent to $\|\cdot \|_{W^{1,2}}$.
\end{corollary}
\begin{proof}
Clearly $\|\cdot \|_\nu \lesssim \| \cdot \|_{W^{1,2}}$ by Theorem \ref{thm:exponential-convergence}. We now prove the reverse inequality. Let $\varphi \in W^{1,2}$ and define $m(\varphi): = \int \varphi \, \omegaFS$. Then $\|\varphi\|_{W^{1,2}} = |m(\varphi)| + \|\del \varphi\|_{L^2}$. By Theorem \ref{thm:exponential-convergence}, we have $$ |\lp \nu, \varphi - m(\varphi) \rp | \lesssim \|\varphi - m(\varphi)\|_{W^{1,2}} = \|\del \varphi\|_{L^2}.$$
Hence $$|m(\varphi)| = |\lp \nu, m(\varphi) \rp|\leq |\lp \nu, \varphi - m(\varphi) \rp | + | \lp \nu, \varphi \rp| \lesssim \|\varphi\|_\nu.$$
This gives  $\| \varphi \|_{W^{1,2}} \lesssim \|\varphi\|_\nu$ and completes the proof.
\end{proof}

\section{Equidistribution of points} \label{sec:equidistribution}

This section is devoted to the proof of Theorem \ref{thm:equidistribution-images}.

We will need the following consequence of Proposition \ref{prop:exponential-estimate}. A proof can be found in \cite{DKW:correspondences}. In what follows, we say that a real valued function $u$ on $\P^1$ is \textit{$(M,\gamma)$ - H\"older continuous} if $|u(x) - u(y)|  \leq M \dist(x,y)^\gamma$ for every $x,y \in \P^1$. When $\gamma=1$ we say that $u$ is\textit{ $M$-Lipschitz}.

\begin{lemma} \label{lemma:holder-exponential}
Let $\mathcal F$ be a bounded subset of $W^{1,2}$. There is a constant $A = A(\mathcal F)> 0$ (independent of $M$ and $\gamma$) such that if $\varphi \in \mathcal F$ is $(M,\gamma)$-H\"older continuous for some constants $M \geq 1$ and $0 < \gamma \leq 1$, then $$\|\varphi\|_\infty \leq A \gamma^{-1}(1+\log M).$$
\end{lemma}

\begin{proof}[Proof of Theorem \ref{thm:equidistribution-images}]
By the Theory of Interpolation between Banach spaces it is enough to prove the result for $\beta = 1$, see \cite{triebel}. We can normalize $\varphi$ so that $\|\varphi\|_{\mathcal C^1} \leq 1$ and $\lp \nu, \varphi \rp = 0$.

Let $\varphi_n := (f^n_\mu)^* \varphi$. Since $$\lp  (f_\mu^n)_*\delta_a , \varphi \rp = \lp \delta_a , (f^n_\mu)^* \varphi \rp = \varphi_n(a)$$
 we need to show that $\|\varphi_n\|_\infty \leq A \gamma^n$ for some constants $A>0$ and $0<\gamma <1$.

\vskip5pt
Let $\lambda_0$ be the norm of $f_\mu^*$ acting on $L^2_{(1,0)}$. By Proposition \ref{prop:normlessthan1}, after replacing $\mu$ by $\mu^{*N}$ for some $N  \geq 1$ if necessary, we may assume that that $0< \lambda_0 <1$. Let $C_n:= e^{\delta_0^n}$ where $\delta_0 > 1$ is a constant such that $1 < \delta_0^{1+\epsilon} < \frac{1}{\lambda_0}$. Set $$\mathcal A^{(n)} := \{(g_1,\ldots,g_n) \in G^n : \|g_n \cdots g_1 \| \leq C_n\}$$ and $$ \mathcal B^{(n)} := \{(g_1,\ldots,g_n) \in G^n : \|g_n \cdots g_1 \| > C_n\}. $$

We can then write $\varphi_n = \varphi^{(1)}_n + \varphi^{(2)}_n$, where $$ \varphi^{(1)}_n (x) := \int_{\mathcal A^{(n)}} \varphi (g_n \cdots g_1 \cdot x) \, \diff \mu^n(g_1,\ldots,g_n)$$ and $$ \varphi^{(2)}_n (x) := \int_{\mathcal B^{(n)}} \varphi (g_n \cdots g_1 \cdot x) \, \diff \mu^n(g_1,\ldots,g_n).$$

We will show separately that $\| \varphi^{(1)}_n\|_\infty$ and $\| \varphi^{(2)}_n\|_\infty$ are bounded by $A \gamma^n$ for some constants $A>0$ and $0<\gamma <1$. 

We start by estimating $ \varphi^{(2)}_n$. Let $M_n = \int (\log \|g_n \cdots g_1\|)^{1+\epsilon} \, \diff \mu^n(g_1,\ldots,g_n)$ be the $(1+\epsilon)$-moment of $\mu^{*n}$. By assumption $M_1$ is finite. We also have that $M_n \leq n^{1+\epsilon} M_1$ by the sub-additivity of $\log\|g\|$. This implies that
\begin{equation} \label{eq:small-probability}
\mu^{\otimes n}(\mathcal B^{(n)})\leq \frac{M_1 \cdot n^{1+\epsilon}}{(\log C_n)^{1+\epsilon}}  = \frac{M_1 \cdot n^{1+\epsilon}}{\delta_0^{n(1+\epsilon)}} \cdot
\end{equation}

Since $\|\varphi\|_\infty \leq 1$, the definition of $\varphi_n^{(2)}$ implies that   $\|\varphi_n^{(2)}\|_\infty \leq M_1 n^{1+\epsilon} \delta_0^{-n(1+\epsilon)}$, which is bounded by $A_2 \gamma_2^n$ for some constants $A_2 > 0$ and $0< \gamma_2 < 1$.
\vskip5pt
In order to estimate $ \varphi^{(1)}_n$ choose a constant $\delta$ such that $\delta_0 < \delta < \delta_0^{1+\epsilon}$ and set $\widehat \varphi_n = \delta^{n} \varphi_n$ and $\widehat \varphi^{(j)}_n = \delta^{n} \varphi^{(j)}_n$, $j=1,2$. We have $\widehat \varphi_n = \widehat \varphi^{(1)}_n +\widehat \varphi^{(2)}_n$.



\vskip10pt
\noindent \textbf{Claim:} $\widehat \varphi_n$, $\widehat \varphi_n^{(1)}$ and $\widehat \varphi_n^{(2)}$ belong to a bounded family in $W^{1,2}$.

\begin{proof}
We will prove that $\widehat \varphi_n$ and $\widehat \varphi_n^{(2)}$ belong to a bounded family. Then the result for $\widehat \varphi_n^{(1)}$ will follow because $\widehat \varphi_n^{(1)} = \widehat \varphi_n - \widehat \varphi_n^{(2)}$.

By the invariance of $\nu$ we have that $\lp \nu, \widehat \varphi_n  \rp = \delta^n \lp \nu, \varphi_n \rp  = \delta^n \lp \nu, \varphi \rp = 0$. We also have that $$\|\del  \widehat \varphi_n \|_{L^2} = \delta^n \| \del \varphi_n \|_{L^2} = \delta^n \| (f^n_\mu)^* \del \varphi \|_{L^2} \lesssim \delta^n \lambda_0^n \|\del \varphi\|_{L^2} \leq (\delta_0^{1+\epsilon})^n \lambda_0^n \|\del \varphi\|_{L^2}$$
is bounded uniformly in $n$ since $\delta < \delta_0^{1+\epsilon} < \frac{1}{\lambda_0}$. Therefore $\widehat \varphi_n$ is a bounded family in $W^{1,2}$ for $n\geq 1$.

We now prove that $\widehat \varphi_n^{(2)}$  belong to a bounded family. Using (\ref{eq:small-probability}) and the definition of $\varphi_n^{(2)}$ we have that $$|\lp  \omegaFS, \varphi_n^{(2)} \rp| \leq \|\varphi_n^{(2)}\|_\infty \leq M_1 n^{1+\epsilon} \delta_0^{-n(1+\epsilon)}$$ and, by Cauchy-Schwarz inequality and Proposition \ref{prop:CS-ineq} $$\|\del \varphi^{(2)}_n\|_{L^2} \leq \|\del \varphi\|_{L^2} M_1 n^{1+\epsilon} \delta_0^{-n(1+\epsilon)} \leq  M_1 n^{1+\epsilon} \delta_0^{-n(1+\epsilon)}.$$

Hence $ \| \widehat \varphi_n^{(2)}\|_{W^{1,2}} \lesssim M_1 n^{1+\epsilon} \delta^n \delta_0^{-n(1+\epsilon)}.$ Since $1<\delta<\delta_0^{1+\epsilon}$, the last quantity is bounded uniformly in $n$. This proves the claim.
\end{proof}

We can now finish the proof of the theorem. Notice that  $ \varphi^{(1)}_n$ is $A_0 C_n^2$-Lipschitz for some universal constant $A_0 >0$. This is not difficult to check using Cartan's decomposition as in Lemma \ref{lemma:jacobian-estimate}. Therefore $ \widehat \varphi^{(1)}_n$ is $A_0 \delta^n C_n^2$-Lipschitz. By Lemma \ref{lemma:holder-exponential} and the above claim we get 
$$\|\widehat \varphi_n^{(1)}\|_{\infty} \leq B \big(1 + \log (A_0 \delta^n C_n^2)\big) = B' \big(1 + n\log \delta  + 2 \delta_0^n \big)$$
for some constants $B,B' > 0$, giving
$$\|\varphi^{(1)}_n\|_{\infty} \leq B' \delta^{-n} \big(1 + n\log (\lambda_0^{-1})  + 2 \delta_0^n \big).$$

Since $1<\delta_0 < \delta$ we get $\|\varphi^{(1)}_n\|_{\infty} \leq A_1 \gamma_1^n$ for some constants $A_1 > 0$ and $0< \gamma_1 < 1$.

\vskip5pt

Taking $A = 2 \max \{A_1,A_2\}$ and  $\gamma = \max \{\gamma_1,\gamma_2\}$ gives $\|\varphi_n\|_{\infty} \leq A \gamma^n$, finishing the proof.
\end{proof} 

\section{Central Limit Theorem} \label{sec:CLT}

This section is devoted to the proof of Theorem \ref{thm:CLT}. We begin by recalling some basic notions, see \cite{bougerol-lacroix} for more details. 

Let $\mu$ be a probability measure on $G= \PSL_2(\C)$ satisfying the first moment condition $\int \log \|g\| \, \diff \mu(g) < + \infty$. Then, the (upper) \textbf{Lyapunov exponent} of $\mu$ is defined as 
\begin{equation}
\gamma:= \lim_{n \to \infty} \frac1n \E(\log\|g_n \cdots g_1\|) = \lim_{n \to \infty} \frac1n \int \log\|g_n \cdots g_1\| \, \diff \mu(g_1) \cdots \diff \mu(g_n).
\end{equation}

It follows from Kingman's subadditive ergodic theorem that 
\begin{equation*}
\gamma= \lim_{n \to \infty} \frac1n \log\|g_n \cdots g_1\| \quad \text{almost surely}
\end{equation*}
and
\begin{equation} \label{eq:lyap-integral-formula}
\gamma = \int_{\P^1} \int_G \log \frac{\|g \cdot v\|}{\|v\|} \diff \mu (g) \diff \nu(x), \quad x=[v].
\end{equation}

Here and in what follows $v$ will denote a non-zero vector in $\C^2$ and $x=[v]$ will be the corresponding point in $\P^1$. We will call $v$ a \textit{lift} of $x$. Notice that the quantity  $\frac{\|g \cdot v\|}{\|v\|}$ is independent of the choice of lift.

For the proof of Theorem \ref{thm:CLT}, we will apply the method of Gordin-Liverani. Recall their theorem.

\begin{theorem}[Gordin-Liverani, \cite{gordin,liverani}]
Let $(X,\mathfrak m)$ be a probability space and let $F:X \to X$. Assume that $\mathfrak m$ is $F$-invariant and ergodic. Let $F^*: \phi \mapsto \phi \circ F$ be the pullback operator acting on $L^2(\mathfrak m)$ and denote by $\Lambda: L^2(\mathfrak m) \to L^2(\mathfrak m)$ its adjoint.

Let $\widetilde \varphi \in L^2(\mathfrak m)$ be such that $\lp \mathfrak m, \widetilde \varphi \rp = 0$ and assume $\widetilde \varphi$ is not a coboundary, that is, not of the form $\widetilde\varphi = \psi \circ F  - \psi$ for some $\psi \in L^2(\mathfrak m)$. If
$$\sum_{n\geq 0} \|\Lambda^n \widetilde \varphi \|^2_{L^2(\mathfrak m)} < + \infty, $$
then the sequence of random variables $Z_n := \frac{1}{\sqrt n} \sum_{j=0}^{n-1} \widetilde \varphi \circ F^j$ converges in distribution to a Gaussian random variable of mean $0$ and variance $\sigma > 0$, where $$\sigma^2 = - \lp \mathfrak m, (\widetilde \varphi)^2\rp + 2 \sum_{n\geq 0} \lp \mathfrak m , \widetilde \varphi \cdot (\widetilde \varphi \circ F^n) \rp.$$
\end{theorem}

Our approach is to first apply the above theorem to a certain dynamical system on $X= G^{\N^*} \times \P^1$ and the observable $\widetilde \varphi(\mathbf g,x) = \log \frac{\|g_1^{-1}v\|}{\|v\|} + \gamma$, where $\mathbf g = (g_1,g_2,\ldots) \in G^{\N^*}$. After that, we will translate the corresponding CLT to the CLT for the random variables $Y^v_n = \log \frac{ \|g_n \cdots g_1 \cdot v\|}{\|v\|}$. \\

Let $\mu$ be a non-elementary probability measure on $G=\PSL_2(\C)$ and denote by $\nu$ the unique $\mu$-stationary measure on $\P^1$. We have the following fundamental result, see \cite[Prop. II.3.3]{bougerol-lacroix}.

\begin{proposition}[Furstenberg] \label{prop:furstenberg} For almost every sequence $\mathbf g = (g_1,g_2,\ldots)$ there exists a point $Z(\mathbf g) \in \P^1$ such that $$\lim_{n \to \infty} (g_1 \cdots g_n)_* \nu = \delta_{Z(\mathbf g)}.$$

Furthermore the distribution of $Z(\mathbf g)$ is $\nu$, that is,
\begin{equation} \label{eq:furstenberg-average}
\int_{G^{\N^*}} \delta_{Z(\mathbf g)} \, \diff \mu^{\N^*}(\mathbf g) = \nu.
\end{equation}
\end{proposition}

An alternative way of phrasing the above result is to say that there exists a map $Z: G^{\N^*} \to \P^1$ defined $\mu^{\N^*}$-almost everywhere such that $Z_* \mu^{\N^*} = \nu$.\\

Let $X:= G^{\N^*} \times \P^1$. Consider the shift map
$$T: G^{\N^*} \to G^{\N^*}, \quad T ((g_1,g_2,\ldots)) = (g_2,g_3,\ldots)$$
and the fibered product
\begin{equation*}
F: X \to  X, \quad F(\mathbf g, x) = (T \mathbf g, g_1^{-1} \cdot x).
\end{equation*}

It follows from Proposition \ref{prop:furstenberg} that
\begin{equation} \label{eq:Z-invariance}
g_1^{-1} Z(\mathbf g) = Z(T \mathbf g) \quad \text{ and } \quad g_1 Z(T \mathbf g) = Z(\mathbf g).  
\end{equation}

In particular, $F$ maps $(\mathbf g,Z(\mathbf g))$ to $(T\mathbf g,Z(T\mathbf g))$. Define a probability measure $\mathfrak m$ on $X$ by
\begin{equation} \label{eq:def-mu-tilde}
\mathfrak m := \int_{G^{\N^*}} \delta_{\mathbf g} \otimes \delta_{Z(\mathbf g)} \, \diff \mu^{\N^*}(\mathbf g) .
\end{equation}

\begin{lemma}
The measure $\mathfrak m$ is $F$-invariant and ergodic.
\end{lemma}

\begin{proof}
The result is well known. The invariance of $\mathfrak m$ follows from (\ref{eq:Z-invariance}) and a direct computation. The ergodicity of $\mathfrak m$ comes from the ergodicity of $\mu^{\N^*}$. See also \cite[p.33]{benoist-quint:book} for a more general statement.
\end{proof}



In what follows we identify functions on $\P^1$ with functions on $G^{\N^*} \times \P^1$ that depend only on the $\P^1$  variable. Similarly, we identify functions on $G \times \P^1$ with functions on $G^{\N^*} \times \P^1$ that depend only on the $\P^1$  variable and the first entry $g_1$ of the sequence $\mathbf g = (g_1, g_2, \ldots)$. Recall that $\Lambda$ is the adjoint of $F^*$ acting on $L^2(\mathfrak m)$.

\begin{lemma} \label{lemma:adjoint}
Let $\psi$ be a function on $G \times \P^1$ viewed as a function on $G^{\N^*} \times \P^1$. Assume that $\psi \in L^2(\mathfrak m)$. Then $\Lambda \psi (\mathbf g, x)$ depends only on $x$ and is given by
\begin{equation} \label{eq:adjoint-formula}
\Lambda \psi (x) = \int_G \psi (h, h \cdot x)\, \diff \mu(h).
\end{equation}

In particular, if $\psi$ is a function on $\P^1$ viewed as a function on $G^{\N^*} \times \P^1$ we have $\Lambda \psi = f_\mu^* \psi$.

\end{lemma}

\begin{proof}
Let $\psi$ be as above. Using (\ref{eq:Z-invariance}) we have, for any $\phi \in L^2(\mathfrak m)$,
\begin{equation*}
\begin{split}
\lp \phi, \Lambda \psi \rp_{L^2(\mathfrak m)} &= \lp F^*\phi, \psi \rp_{L^2(\mathfrak m)} = \int_{G^{\N^*}} \phi(T \mathbf g, g_1^{-1} Z(\mathbf g)) \, \psi( g_1 , Z(\mathbf g)) \diff \mu^{\N^*}(\mathbf g) \\ &= \int_{G^{\N^*}} \phi(T \mathbf g,  Z( T \mathbf g)) \psi(g_1, g_1 Z(T\mathbf g)) \diff \mu^{\N^*}(\mathbf g) \\ &= \int_{G^{\N^*}}  \phi(T \mathbf g,  Z( T \mathbf g)) \left( \int_G \psi(g_1, g_1 Z(T\mathbf g) \diff \mu(g_1) \right) \diff \mu^{\N^*}(T \mathbf g) \\ &=  \int_{G^{\N^*}}  \phi(\mathbf g',  Z( \mathbf g')) \left( \int_G \psi(g_1, g_1 Z(\mathbf g') \diff \mu(g_1) \right) \diff \mu^{\N^*}(\mathbf g') \\ &= \Big \lp \phi, \int_G \psi(g_1, g_1 \cdot x) \diff \mu(g_1) \Big \rp_{L^2(\mathfrak m)},
\end{split}
\end{equation*} 
where on the second to last step we used the change of coordinates $\mathbf g'= T \mathbf g$ and the fact that $\mu^{\N^*}$ is $T$-invariant. Since $\phi \in L^2(\mathfrak m)$ is arbitrary, this proves (\ref{eq:adjoint-formula}). The final statement is straightforward. This proves the lemma.
\end{proof}

\begin{lemma} \label{lemma:m=nu}
If $\psi \in L^2(\mathfrak m)$ depends only on the $\P^1$ variable then  $ \lp \mathfrak m ,\psi \rp = \lp \nu ,\psi \rp$. In particular, for such $\psi$ we have $\| \psi\|_{L^p(\mathfrak m)} = \|\psi\|_{L^p(\nu)}$ for $p=1$ or $2$.
\end{lemma}
\begin{proof}
From the definition of $\mathfrak m$ and the fact that $Z_*\mu^{\N^*} = \nu$ (cf. eq. (\ref{eq:furstenberg-average})), it follows that 
$$\int_{G^{\N^*} \times \P^1} \psi (x) \, \diff \mathfrak m (\mathbf g,x)= \int_{G^{\N^*}} \psi(Z(\mathbf g)) \, \diff \mu^{\N^*}(\mathbf g) = \int_{\P^1} \psi(x) \, \diff \nu(x).$$

This gives us the first assertion. Similar identities for $|\psi|$ and $|\psi|^2$ give the second assertion.
\end{proof}

Consider now the function
\begin{equation} \label{eq:def-phi}
\varphi: G^{\N^*} \times \P^1 \to \R, \quad \varphi(\mathbf g, x) = \log \frac{\|g_1^{-1}\cdot v\|}{\|v\|}, \quad  x = [v].
\end{equation}

Notice that $\varphi \circ F^j(\mathbf g, x) = \varphi(T^j \mathbf g, g_j^{-1} \cdots g_1^{-1}\cdot x)= \log \frac{\|g_{j+1}^{-1} g_j^{-1} \cdots g_1^{-1}\cdot v\|}{\|g_j^{-1} \cdots g_1^{-1} \cdot v\|}$ . So,  the associated Birkhoff sum is
\begin{equation} \label{eq:birkhoff}
\sum_{j=0}^{n-1} \varphi \circ F^j (\mathbf g,x) = \log \frac{\|g_n^{-1} \cdots g_1^{-1} \cdot v\|}{\|v\|} , \quad x =[v].
\end{equation}

\begin{lemma} \label{lemma:average-phi}
We have $\lp \mathfrak m, \varphi \rp = -\gamma$.
\end{lemma}
\begin{proof}
Let $Z(\mathbf g) \in \P^1$ be as in Proposition \ref{prop:furstenberg}. Let $W(\mathbf g) \in \C^2\setminus \{0\}$ be a lift of  $Z(\mathbf g)$. By (\ref{eq:Z-invariance}) we can chose $W$ so that $g_1^{-1} W(\mathbf g) = W(T \mathbf g)$. Using the definition of $\mathfrak m$, the fact that $\mu^{\N^*}$ is invariant by $T$ and equations (\ref{eq:furstenberg-average}) and (\ref{eq:lyap-integral-formula}) we get 
\begin{align*}
\lp \mathfrak m, \varphi \rp &= \int_{G^{\N^*}} \varphi(Z(\mathbf g)) \, \diff \mu^{\N^*} (\mathbf g) = \int_{G^{\N^*}} \log \frac{\|g_1^{-1} W(\mathbf g)\|}{\|W(\mathbf g)\|} \, \diff \mu^{\N^*} (\mathbf g) \\
&= \int_{G^{\N^*}} \log \frac{\| W(T \mathbf g)\|}{\|g_1 W( T \mathbf g)\|} \, \diff \mu^{\N^*} (\mathbf g) = \int_G \int_{G^{\N^*}} \log \frac{\| W( \mathbf g')\|}{\|g_1 W(\mathbf g')\|} \, \diff \mu^{\N^*} (\mathbf g') \, \diff \mu(g_1) \\
& = \int_G \int_{\P^1} \log \frac{\|v\|}{\|g_1 \cdot v\|} \diff \nu(x) \diff \mu(g_1) = - \gamma.
\end{align*}

The lemma follows.
\end{proof}

\begin{proposition} \label{prop:gordin-condition}
Let $\varphi$ be the function in (\ref{eq:def-phi}). Then $\widetilde \varphi = \varphi - \lp \mathfrak m, \varphi \rp$ belongs to $L^2(\mathfrak m)$ and satisfies Gordin's condition.  Namely, $$\sum_{n \geq 0} \|\Lambda^n \widetilde \varphi\|^2_{L^2(\mathfrak m)} < + \infty.$$
\end{proposition}

\begin{proof}
Let us first check that $\widetilde \varphi \in L^2(\mathfrak m)$. Let $W$ be a lift of $Z$ as in the proof of Lemma \ref{lemma:average-phi}. We have
\begin{align*}
\lp \mathfrak m, |\varphi|^2 \rp &= \int_{G^{\N^*} \times \P^1}   \left( \log \frac{\|g_1^{-1}\cdot v\|}{\|v\|} \right)^2 \diff \mathfrak m (\mathbf g, x) = \int_{G^{\N^*}} \left( \log \frac{\|g_1^{-1}\cdot W(\mathbf g)\|}{\|W(\mathbf g)\|} \right)^2\diff  \mu^{\N^*}(\mathbf g) \\
 & \leq \int_G \sup_{x \in \P^1, [v]=x} \left( \log \frac{\|g_1^ {-1}\cdot v\|}{\|v\|} \right)^2 \diff  \mu(g_1) = \int_G (\log \|g_1\|)^2 \, \diff \mu(g_1) < + \infty,
\end{align*}
where we have used that $\|g\| = \|g^{-1}\|$ for $g \in \SL_2(\C)$ and the assumption that $\mu$ has a finite second  moment. So $\varphi \in L^2(\mathfrak m )$, which implies that $ \widetilde \varphi \in L^2(\mathfrak m )$ as claimed.

Let us now prove Gordin's estimate. We begin by noticing that $\widetilde \varphi(\mathbf g,x)$ depends only on the first entry of $\mathbf g$, so we may apply Lemma \ref{lemma:adjoint}.  Then $$\psi (x) : = \Lambda \widetilde \varphi (x) = \int_G \log \frac{\|v\|}{\|g \cdot v \|} \diff \mu (g) + \gamma, \quad x = [v]$$ depends only on the $\P^1$ variable and $$\Lambda^n \widetilde \varphi = \Lambda^{n-1} \psi = (f_\mu^*)^{n-1} \psi.$$

We claim that $\psi \in W^{1,2}$. In order to see that, define $\theta_g(x) := \log \frac{\|g \cdot v\|}{\|v\|}$, $x=[v]$. Then $\psi (x) = \int_G - \theta_g(x) \diff \mu (g) + \gamma$.  Now, each $\theta_g$ is a smooth function and we have from Lemma \ref{lemma:theta_g-estimate} in the appendix that $\|\theta_g\|_{W^{1,2}} \lesssim 1 + \log \|g\|$. Then $$\|\psi\|_{W^{1,2}} \leq \int_G \|\theta_g\|_{W^{1,2}} \diff \mu(g) + \gamma \lesssim \int_G (1+\log \|g\|) \diff \mu (g) + \gamma < + \infty,$$ showing that $\psi \in W^{1,2}$.

Now, from Lemma \ref{lemma:m=nu} and the invariance of $\mathfrak m$ we get $$\lp \nu, \psi \rp = \lp \mathfrak m ,\psi \rp = \lp \mathfrak m ,\Lambda \widetilde \varphi \rp = \lp \mathfrak m,  \widetilde \varphi \rp = 0.$$ This can also be checked directly using (\ref{eq:lyap-integral-formula}) and the expression of $\psi$.

From Theorem \ref{thm:exponential-convergence} we have that $(f_\mu^*)^{n-1} \psi$ converges to $\lp \nu, \psi \rp = 0$ in $W^{1,2}$ exponentially fast. Since, also by Theorem \ref{thm:exponential-convergence}, $\nu$ acts continuously on $W^{1,2}$ and $\|\, |h|\, \| _{W^{1,2}} \lesssim \|h\| _{W^{1,2}}$ for $h \in W^{1,2}$ (cf. \cite[Prop. 4.1]{dinh-sibony:decay-correlations}) we get $$\|\Lambda^n \widetilde \varphi \|_{L^1(\nu)} = \|(f_\mu^*)^{n-1} \psi\|_{L^1(\nu)} \lesssim \lambda^n$$ for some constant $0 < \lambda < 1$.

Observe now that, using Lemma \ref{lemma:adjoint}
\begin{align*}
\Lambda^n \varphi (x) &= (f_\mu^*)^{n-1} \int_G \varphi(g_1,g_1 \cdot x) \diff \mu(g_1) = (f_\mu^*)^{n-1} \int_G \log \frac{\|v\|}{\|g_1 \cdot v\|} \diff \mu(g_1) \\ &= \int_{G^n} \log \frac{\|g_2 g_3 \cdots g_n \cdot v\|}{\|g_1 g_2 \cdots g_n \cdot v \|} \diff\mu(g_1) \cdots \diff \mu(g_n).
\end{align*}

Hence $\|\Lambda^n \varphi\|_\infty \leq \int_G \log \|g_1\| \diff \mu(g_1)$ for every $n \geq 1$. In particular, there is a constant $C$ such that $\|\Lambda^n \widetilde \varphi\|_{\infty} \leq C$ for every $n \geq 1$. 

By interpolating between the spaces $L^\infty(\nu) \subset L^2(\nu) \subset L^1(\nu)$ we conclude that
 $$\|\Lambda^n \widetilde \varphi\|_{L^2(\mathfrak m)} = \|\Lambda^n \widetilde \varphi\|_{L^2(\nu)} \lesssim  \|\Lambda^n \widetilde \varphi\|_{L^\infty(\nu)}^{1 \slash 2}  \|\Lambda^n \widetilde \varphi\|_{L^1(\nu)}^{1 \slash 2}  \lesssim \lambda^{n \slash 2},$$ which gives  $\sum_{n \geq 0} \|\Lambda^n \widetilde \varphi\|^2_{L^2(\mathfrak m)} < + \infty$. The proof is complete.
\end{proof}

\begin{lemma} \label{lemma:not-coboundary}
The function $\widetilde \varphi = \varphi - \lp \mathfrak m,\varphi \rp$ is not a coboundary.
\end{lemma}

\begin{proof}
Assume by contradiction that $\widetilde \varphi = \psi \circ F - \psi$ for some $\psi \in L^2(\mathfrak m) $. Then $\widetilde \varphi \circ F^j = \psi \circ F^{j+1} - \psi \circ F^j$ for $j \geq 0$ and 
\begin{equation} \label{eq:coboundary-sum}
\sum_{j=0}^{n-1} \widetilde \varphi \circ F^j = \psi \circ F^n - \psi.
\end{equation}

The $L^2$ norm of the right-hand side of (\ref{eq:coboundary-sum}) is bounded by $\|\psi \circ F^n\|_{L^2(\mathfrak m)} + \|\psi\|_{L^2(\mathfrak m)} = 2\|\psi\|_{L^2(\mathfrak m)}$. In particular, this quantity is bounded independently of $n$.

We will now show that the $L^2$ norm of the left-hand side of (\ref{eq:coboundary-sum}) is unbounded as $n$ goes to infinity. This contradiction will end the proof. 

From (\ref{eq:birkhoff}) we have that $$\sum_{j=0}^{n-1} \widetilde \varphi \circ F^j (\mathbf g,x) = \log \frac{\|g_n^{-1} \cdots g_1^{-1} \cdot v\|}{\|v\|} + n \gamma, \quad x = [v].$$

Let $W(\mathbf g)$ be as in the proof of Lemma \ref{lemma:average-phi}.  Then,
\begin{equation} \label{eq:sum-phi-F^j}
\begin{split}
\Big \| \sum_{j=0}^{n-1} \widetilde \varphi \circ F^j \Big \|^2_{L^2(\mathfrak m)} &= \int \left(   \log \frac{\|g_n^{-1} \cdots g_1^{-1} \cdot W(\mathbf g) \|}{\|W(\mathbf g)\|} + n \gamma \right)^2 \diff \mu^{\N^*}(\mathbf g) \\ &= \int \left(   \log \frac{\|W(T^n \mathbf g) \|}{\|g_n \cdots g_1 W(T^n \mathbf g)\|} + n \gamma \right)^2 \diff \mu^{\N^*}(\mathbf g) \\ &= \int \left( -  \log \frac{\|g_n \cdots g_1 \cdot W(\mathbf g') \|}{\|W(\mathbf g')\|} + n \gamma \right)^2 \diff \mu^{\N^*}(\mathbf g') \, \diff \mu(g_1) \cdots \diff \mu(g_n) \\ &= \int \left( -  \log \frac{\|g_n \cdots g_1 \cdot v \|}{\| v \|} + n \gamma \right)^2 \diff \nu(x)\, \diff \mu(g_1) \cdots \diff \mu(g_n).
\end{split}
\end{equation}

Let $\zeta_n$ be the random variable $\log \frac{\|g_n \cdots g_1 \cdot v\|}{\|v\|} - n\gamma$ on $G^{\N^*} \times \P^1$, where $x = [v]$ has law $\nu$ and the $g_i$ have law $\mu$. Notice that the last integral in (\ref{eq:sum-phi-F^j}) is the variance of $\zeta_n$. Hence, in order to prove that (\ref{eq:sum-phi-F^j}) is unbounded it is enough to show that the sequence of the distributions of  $\zeta_n$ is not tight (that is, not relatively compact in the space of probability measures on $\R$, see \cite{billingsley}).

It follows from \cite[V.8.5 and V.8.6]{bougerol-lacroix} that for every fixed $v \in \C^2 \setminus \{0\}$ and any $ c> 0$ we have $$\lim_{n \to \infty} \frac{1}{n} \sum_{k=1}^n \mathbf P_x \Big( \log \frac{\|g_k \cdots g_1 \cdot v\|}{\|v\|} -k \gamma < - c \Big) = 1,$$ where $\mathbf P_x$ denotes the probability with respect to $\mu^{\N^*} \otimes \delta_x$. Using Fubini's Theorem and Lebesgue's Dominated Convergence Theorem we get $$\lim_{n \to \infty} \frac{1}{n} \sum_{k=1}^n \mathbf P \Big( \log \frac{\|g_k \cdots g_1 \cdot v\|}{\|v\|} - k \gamma < - c \Big) = 1,$$ where $\mathbf P$ denotes the probability with respect to $\mu^{\N^*} \otimes \nu$. This implies that the sequence of the distributions of  $\zeta_n$ is not tight, thus finishing the proof.
\end{proof}

We will need the next proposition that shows that for most sequences $\mathbf g$ the quantities $\|g_1 \ldots g_n\|$ and $\frac{\|g_1 \cdots g_n \cdot v\|}{\|v\|}$ are comparable  for any given $v\in \C^2 \setminus \{0\}$. See  \cite[III.3.2]{bougerol-lacroix}  and  \cite[Rmk. 4.26]{benoist-quint:book}.

\begin{proposition} \label{prop:norm-comparison}
Let $\mu$ be a non-elementary probability measure on $\PSL_2(\C)$. Then for any $\varepsilon > 0$ there exists a $\delta > 0$ such that, for every non-zero $v \in \C^2$  \begin{equation}
\mu^{\N^*} \left\lbrace \mathbf g = (g_1,g_2,
\ldots) : \delta \leq \frac{\|g_n \cdots g_1 \cdot v\|}{\|g_n \cdots g_1\| \|v\|} \leq 1 \quad \text{for all } n \geq 1 \right \rbrace \geq 1 - \varepsilon.  
\end{equation}
\end{proposition}

\begin{proof}[Proof of Theorem \ref{thm:CLT}]
Consider the dynamical system $F: X \to X$, the measure $\mathfrak m$ on $X$ and the function $\widetilde \varphi$ introduced above. By Proposition \ref{prop:gordin-condition}, Lemmas \ref{lemma:average-phi} and \ref{lemma:not-coboundary} we can apply Gordin-Liverani's Theorem to $\widetilde \varphi$ and $F$. This gives that the sequence of random variables $Z_n = \frac{1}{\sqrt n} \sum_{j=0}^{n-1} \widetilde \varphi \circ F^j$ converges in distribution to a Gaussian random variable of mean zero and variance $\sigma > 0$.

Let $Y_n:=\frac{1}{\sqrt n} (\log \frac{\|g_n \cdots g_1 \cdot v\|}{\|v\|} - n\gamma)$ be random variables on $G^{\N^*} \times \P^1$, where $x = [v]$ has law $\nu$ and the $g_i$ have law $\mu$. We claim that $Z_n$ and $-Y_n$ have the same distribution. Since the Gaussian law is symmetric around the origin, the convergence of $Z_n$ to the normal distribution will give the convergence of $Y_n$ to the same distribution. In order to do so, we compare the characteristic functions $\chi_{Z_n}(t)$ and $\chi_{Y_n}(t)$ of $Z_n$ and $Y_n$.

From (\ref{eq:birkhoff}) and Lemma \ref{lemma:average-phi} we have  $Z_n = \frac{1}{\sqrt n }(\log \frac{\|g_n^{-1} \cdots g_1^{-1} \cdot v\|}{\|v\|} + n \gamma)$.   Then
\begin{equation*}
\begin{split}
\chi_{Z_n}(t) &= \mathbf E(e^{it Z_n}) = \int e^{ \frac{it}{\sqrt n }(\log \frac{\|g_n^{-1} \cdots g_1^{-1} \cdot v\|}{\|v\|} + n \gamma)} \diff \mathfrak m (\mathbf g,x) \\ &=  \int e^{ \frac{it}{\sqrt n }(\log \frac{\|g_n^{-1} \cdots g_1^{-1} \cdot W(\mathbf g)\|}{\|W(\mathbf g)\|} + n \gamma)} \diff \mu^{\N^*} (\mathbf g)  \\ &= \int e^{\frac{it}{\sqrt n}(-\log \frac{\|g_n \cdots g_1 \cdot v \|}{\| v \|} + n \gamma)} \diff \nu(x)\, \diff \mu(g_1) \cdots \diff \mu(g_n),
\end{split}
\end{equation*}
where in the last step we used the same argument as in (\ref{eq:sum-phi-F^j}).

On the other hand
\begin{equation*}
\begin{split}
\chi_{-Y_n}(t) &= \mathbf E(e^{-it Y_n}) = \int e^{-it Y_n}\diff \nu(x)\, \diff \mu(g_1) \cdots \diff \mu(g_n) \\ &= \int e^{\frac{it}{\sqrt n }(-\log \frac{\|g_n \cdots g_1 \cdot v\|}{\|v\|} + n \gamma)} \diff \nu(x)\, \diff \mu(g_1) \cdots \diff \mu(g_n) = \chi_{Z_n}(t).
\end{split}
\end{equation*}

As the characteristic function of a random variable determines its distribution, we conclude that $Z_n$ and $-Y_n$ have the same distribution.

By the above remarks, the sequence of random variables $Y_n$ on $G^{\N^*} \times \P^1$ converges in law to $\mathcal N(0;\sigma^2)$. From Proposition \ref{prop:norm-comparison} we conclude that for any nonzero $v \in \C^2$ the sequence of random variables $Y_n^ v := \frac{1}{\sqrt n} \big (\log \frac{\|g_n \cdots g_1 \cdot v\|}{\|v\|} - n \gamma\big)$ on $G^{\N^*}$ converges in law to $\mathcal N(0;\sigma^2)$. The proof is now complete.
\end{proof}

\section{Regularity of the stationary measure} \label{sec:regularity}

We now study the regularity of stationary measures. Throughout this section $\mu$ will be a non-elementary probability measure on $G = \PSL_2(\C)$ such that $\int_G \log\|g\|\,\diff \mu(g) <+\infty$ and $\nu$ will denote the unique $\mu$-stationary measure. We will also replace $\mu$ by $\mu^{*N}$ for some $N\geq 1$ when necessary and assume that the norm of $f_\mu^*$ acting on $L^2_{(1,0)}$ is strictly less than one (cf. Propoistion \ref{prop:normlessthan1}). Notice that $\mu$ and $\mu^{*N}$ have the same stationary measure.

As we will see, the regularity of $\nu$ will depend on the moments of $\mu$. We'll need the following notion.

\begin{definition} \label{def:moments}
Let $\chi: [0,+\infty) \to [0,+\infty)$ be a non-negative function. The $\chi$-moment of a probability measure $\mu$ on $G = \PSL_2(\C)$ is the number
$$\int_G \chi(\log\|g\|) \, \diff \mu(g) \in \R_{\geq 0} \cup \{+\infty \}.$$
If the above integral is finite, we say that $\mu$ satisfies the $\chi$-moment condition, or equivalently, that $\mu$ has a finite $\chi$-moment.

In particular, if $\chi(s)=s^p$ (resp. $\chi(s)=e^{ps}$) for some $p>0$ we say that $\mu$ satisfies the $p^{th}$-moment condition (resp. an exponential moment condition).
\end{definition}

Recall that we are assuming that $\mu$ has a finite first moment. In particular, if $\chi(s) \lesssim s$ for $s$ large, then $\mu$ satisfies the $\chi$-moment condition. Hence, we'll often assume $\chi(s) \gtrsim s$ for $s$ large. It is also natural to consider $\chi$ convex and increasing. In that case, it follows from the sub-additivity of $\log\|g\|$ that if $\mu$ has a finite $\chi$-moment then $\mu^{*n}$ has a finite $\chi_n$-moment, where $\chi_n(s) := \chi(\frac1n s)$ for $n \geq 1$. In particular, if $\mu$ has a finite $p^{th}$ moment or a finite exponential moment then the same is true for $\mu^{*n}$. 

\vskip5pt

We now introduce a notion of regularity for probability measures following the theory of super-potentials, \cite{dinh-sibony:acta}.

Consider the unit ball in $W^{1,2}$
$$\B:=\big\{\varphi\in W^{1,2} : \|\varphi\|_{W^{1,2}}\leq 1\big\}.$$
 
Let $\| \cdot \|$ be an auxiliary norm on $W^{1,2}$ and denote by $\dist$ the distance induced by $\| \cdot \|$. We will be interested in norms that are weaker than $\| \cdot \|_{W^{1,2}}$. 

\begin{definition}
Let $m$ be a probability measure on $\P^1$. We say that $m$ has a  H\"older continuous super-potential with respect to $W^{1,2}$ and $\dist$ if the restriction of $m$ to 
$\B$ is a H\"older continuous function with respect to $\dist$. 
\end{definition}

The functional on $W^{1,2}$ defined by $m$ is a kind of superpotential of $m$ (compare with \cite{dinh-sibony:acta}). Notice that the above notion doesn't change if we replace $\B$ by any bounded open subset of $W^{1,2}$. In particular, we can replace $\B$ by the unit ball of $W^{1,2}$ with respect to any norm on $W^{1,2}$ that is equivalent to $\|\cdot\|_{W^{1,2}}$.

It will be convenient to work with the following norm and corresponding ball:
$$\|\varphi\|_\nu:=|\langle \nu,\varphi\rangle|+ \|\partial\varphi\|_{L^2} \quad \text{and} \quad 
\B_\nu:=\big\{\varphi\in W^{1,2}: \|\varphi\|_\nu\leq 1\big\}.$$

It follows from Corollary \ref{cor:nu-norm} that $\|\cdot\|_\nu$ is equivalent to $\|\cdot\|_{W^{1,2}}$. For later use, define also
$$\B^0_\nu:=\big\{\varphi\in \B_\nu : \langle \nu,\varphi\rangle =0\big\}$$
and $$\Lambda:= f_\mu^*: W^{1,2} \longrightarrow W^{1,2}.$$

Since $\nu$ is stationary, $\B_\nu$ and $\B_\nu^0$ are invariant by $\Lambda$. Moreover, $\frac12 (\Lambda-\id)$ maps $\B_\nu$ to $\B_\nu^0$ and by Proposition \ref{prop:normlessthan1} there is a constant $\delta>1$ such that $\B_\nu^0$ is invariant by
$\widetilde\Lambda:=\delta \Lambda$.

\vskip5pt

Denote by $\DD(a,r)$ the disc of center $a$ and of radius $r$ in $\P^1$.  Fix $0<\epsilon_0 < \frac12$. For $0<\epsilon< \epsilon_0$, $0<r\leq 1$ and $a\in\P^1$, set 
$$u^\epsilon_{a,r}(z):=\max\Big(-\log \frac{\dist_{\P^1}(z,a)}{2r},0\Big)^{\frac12 -\epsilon}.$$

Then $u^\epsilon_{a,r}$ belongs to $W^{1,2}$ and it is supported by $\DD(a,2r)$. One can also check that the $u^\epsilon_{a,r}$ belong to a bounded subset of $W^{1,2}$, see \cite[Ex. 2]{vigny:dirichlet}. 

Define 
$$\mathcal V_\epsilon(r):=\max_{a\in\P^1} \|u^\epsilon_{a,r}\| < + \infty.$$

\begin{proposition} \label{p:size-of-disc}
Assume that $m$ has a H\"older continuous super-potential with respect to $W^{1,2}$ and $\dist$. Then there are constants $c>0$ and $\alpha>0$ independent of $\epsilon$ such that for $0<r\leq 1$
$$m(\DD(a,r)) \leq c \, \mathcal V_\epsilon(r)^\alpha.$$
\end{proposition}
\begin{proof}
Notice that $u^\epsilon_{a,r} \geq d$ on $\DD(a,r)$ where $d:= \sqrt{\log 2}$. Since $m$ has a H\"older continuous super-potential, we have 
$$m(\DD(a,r)) \leq d^{-1} |\langle m, u^\epsilon_{a,r}\rangle| = d^{-1} |m(u^\epsilon_{a,r}) -m(0)| \leq c \|u^\epsilon_{a,r}-0\|^\alpha  \leq c\,  \mathcal V_\epsilon(r)^\alpha,$$ for some positive constants $c$ and $\alpha$. This ends the proof.
\end{proof} 

\begin{proposition} \label{p:sp}
Let $\mu$ be a non-elementary probability measure on $G=\PSL_2(\C)$ having a finite first moment. Assume that $\Lambda = f^*_\mu: W^{1,2}\to W^{1,2}$ is bounded with respect to the norm $\|\cdot\|$ with $\|\cdot\|_{L^1}\leq c \|\cdot\|$ and $\|\cdot\| \leq c \|\cdot\|_{W^{1,2}}$ for some constant $c>0$.
Then $\nu$ has a H\"older continuous super-potential with respect to $W^{1,2}$ and the distance $\dist$. 
\end{proposition} 

We will need the following lemma.

\begin{lemma} \label{l:Holder}
Let $K$ be a metric space. Let $A\geq 1$ be a constant and let $F_n:K\to K$ be a sequence of Lipschitz maps on $K$ such that $\|F_n\|_{\Lip}\leq A^n$ for every $n$. Then for any bounded H\"older continuous function $\vartheta:K\to \C$ and any $ \delta > 1$, the function 
$$\sum_{n\geq 0} \delta^{-n} \, (\vartheta \circ F_n)$$
is also H\"older continuous. If furthermore $K$ has finite diameter, then the assumption on the boundedness of $\vartheta$ is superfluous.  
\end{lemma}
\begin{proof}
In the particular case where $F_n = F^n$ for some Lipschitz map $F$ this is Lemma 1.19 in \cite{dinh-sibony:cime}. It can be easily checked that the proof given there extends to the present setting.
\end{proof}

\proof[Proof of Proposition \ref{p:sp}]
We apply Lemma \ref{l:Holder} to $K:=\B_\nu$, $F_n:=\widetilde\Lambda^n\circ\big(\frac12 (\Lambda-\id)\big)$ and $\vartheta$ the restriction of 
$\omegaFS$ to $\B_\nu$. Recall that $\widetilde\Lambda = \delta\Lambda$ for some $\delta>1$ and that both $\widetilde\Lambda$ and $\frac12 (\Lambda - \id)$ preserve $\B_\nu$.

Since $ \|\cdot \|_{L^1} \lesssim \|\cdot \|$ by hypothesis, we have $|\vartheta (\varphi)| \leq \|\varphi\|_{L^1} \lesssim \|\varphi\|$ for $\varphi \in \B_\nu$. Hence $\vartheta$ is a Lipschitz function on $K$. Moreover, since $\Lambda = f_\mu^*$ is bounded with respect to $\|\cdot\|$ by assumption, the maps $\widetilde\Lambda$ and $\frac12(\Lambda-\id)$ are also Lipschitz on $K$. So we have $\|F_n\|_{\Lip} \leq A^n$ for some constant $A\geq 1$. Notice also that the assumption $\|\cdot\| \lesssim \|\cdot\|_{W^{1,2}}$ implies that $\B_\nu$ has finite diameter with respect to $\dist$.

We now have, for $\varphi \in K$  $$2 \, \delta^{-n} \, \vartheta \circ F_n (\varphi) = \vartheta \circ \Lambda^n \circ (\Lambda - \id) (\varphi) = \lp \omegaFS, \Lambda^{n+1}(\varphi) - \Lambda^n(\varphi) \rp = \lp (f_\mu^{n+1})_* \omegaFS - (f_\mu^n)_* \omegaFS, \varphi \rp . $$

It follows from Theorem \ref{thm:exponential-convergence} that $\lim_{n\to\infty} (f_\mu^n)_*(\omegaFS) = \nu$. Therefore, 
$$2\sum_{n\geq 0} \delta^{-n} \, (\vartheta \circ F_n )= -\omegaFS+\lim_{n\to\infty} (f_\mu^n)_*(\omegaFS)=-\omegaFS+\nu$$

By Lemma \ref{l:Holder} we get that $-\omegaFS+\nu$ defines a H\"older continuous function on $\B_\nu$. It follows that $\nu$ defines a H\"older continuous function on $\B_\nu$.
The proof is complete.
\endproof

We now apply the above results for some choices of the norm $\|\cdot\|$. Consider a Young's function $\Phi:[0,\infty)\to [0,\infty)$, that is, a convex increasing function such that
$$\lim_{t\to 0} \frac{\Phi(t)}{t} =0\quad \text{and} \quad  \lim_{t\to \infty} \frac{\Phi(t)}{t} =\infty.$$
We also assume that $e^{-t^2}\Phi(t)$ is bounded.
Consider the Luxemburg norm (or gauge norm) $$\|\varphi\|_\Phi := \inf\Big\{ A\in[0,\infty) : \int_{\P^1} \Phi \Big( \frac{|\varphi|}{A} \Big) \omegaFS \leq 1 \Big\}$$ and the associated Birnbaum-Orlicz space $L_\Phi(\P^1)$ consisting of measurable functions on $\P^1$ having finite $\|\cdot\|_\Phi$ norm, see \cite{rao-ren}.

The distance associated with this norm is denoted by $\dist_\Phi$. Since we are assuming that $e^{-t^2}\Phi(t)$ is bounded we have, by Moser-Trudinger's estimate (Proposition \ref{prop:exponential-estimate}), that $\|\cdot\|_\Phi \lesssim \|\cdot\|_{W^{1,2}}$ and $W^{1,2}\subset L_\Phi(\P^1)$. 

Define also the function $\eta_\Phi:\R_{\geq 0}\to [0,+\infty]$ by 
$$\eta_\Phi(s):=\sup_{\varphi\in W^{1,2}\setminus\{0\}} \frac{\|\varphi\|_{e^s \Phi} }{\|\varphi\|_{\Phi}} \cdot$$

\begin{theorem} \label{t:general-sp}
Let $\mu$ be a non-elementary probability measure on $G=\PSL_2(\C)$ having a finite $\chi$-moment and let $\nu$ be the associated stationary measure. Assume that $\eta_\Phi(4s)\lesssim \chi(s)+1$ for $s\geq 0$. Then $\nu$ has a H\"older continuous super-potential with respect to $W^{1,2}$ and the distance $\dist_\Phi$. 
\end{theorem}
\proof
It is well-known that $\|\cdot\|_{L^1}\lesssim \|\cdot\|_\Phi$ (see \cite{rao-ren}) and we have seen that $\|\cdot\|_\Phi \lesssim \|\cdot\|_{W^{1,2}}$. So by Proposition \ref{p:sp}, it is enough to check that $\Lambda: W^{1,2}\to W^{1,2}$ is bounded with respect to the norm $\|\cdot\|_\Phi$. We have for $\varphi\in W^{1,2}$
$$\|\Lambda\varphi\|_\Phi = \Big\|\int_G g^*\varphi \, \diff\mu(g)\Big\|_\Phi \leq \int_G \|g^*\varphi\|_\Phi \, \diff \mu(g).$$
Since, by assumption, $\mu$ has a finite $\chi$-moment, it is enough to show that 
$$\|g^*\varphi\|_\Phi \lesssim (\chi(\log\|g\|) + 1) \|\varphi\|_\Phi \quad \text{for every } g \in G.$$

Set $s:=\log\|g\|$. Recall from Lemma \ref{lemma:jacobian-estimate} that $g_* \omegaFS \leq \|g\|^4 \omegaFS$. Then, for any $A > 0$, we have
$$\int_{\P^1} \Phi \Big( \frac{|g^*\varphi|}{A} \Big)\omegaFS = \int_{\P^1} \Phi \Big( \frac{|\varphi|}{A} \Big) g_*\omegaFS \leq  \int_{\P^1} \|g\|^4 \Phi \Big( \frac{|\varphi|}{A} \Big) \omegaFS =  \int_{\P^1}e^{4s} \Phi \Big( \frac{|\varphi|}{A} \Big) \omegaFS.$$
Hence
$$\|g^*\varphi\|_\Phi \leq \|\varphi\|_{e^{4s}\Phi} \leq \eta_\Phi(4s) \|\varphi\|_\Phi \lesssim (\chi(s) + 1) \|\varphi\|_\Phi = (\chi(\log\|g\|) +1) \|\varphi\|_\Phi.$$
The theorem follows.
\endproof

We can now use Theorem \ref{t:general-sp} to obtain explicit regularity properties of $\nu$ in terms of the moments of $\mu$. The idea is the following: assuming that $\mu$ has a finite $\chi$-moment, find a suitable Young's function $\Phi$ so that $\eta_\Phi(4s)\lesssim \chi(s)+1$.  Theorem \ref{t:general-sp} will then give that $\nu$ has a H\"older continuous super-potential with respect to $\dist_\Phi$. Together with Proposition \ref{p:size-of-disc} this will give an estimate for the mass of $\nu$ on small discs.

\vskip5pt

The following corollaries illustrate two extremal cases where our method applies. The same idea can be extended to other moment conditions on $\mu$.

\begin{corollary} \label{cor:reg-expmoment}
Let $\mu$ be a non-elementary measure on $\PSL_2(\C)$ with finite exponential moment and let $\nu$ be the associated stationary measure. Then there is a number $q \in [1,\infty)$ such that $\nu$ has a H\"older continuous super-potential with respect to $W^{1,2}$ and the $L^q$-norm. In particular, there are constants $\theta>0, A>0, c>0$ and $\alpha>0$ such that 
$$\int_{\P^1} e^{\theta |\varphi|^2} \, \diff \nu \leq A \quad \text{and} \quad \nu(\DD(a,r))\leq c r^\alpha$$
for every $\varphi \in W^{1,2}$ with $\|\varphi\|_{W^{1,2}}\leq 1$, $a\in \P^1$ and $0<r\leq 1$.
\end{corollary}
\proof
Fix a number $q$ large enough and choose $\Phi(t)=t^q$. It can be easily seen that  $\|\cdot\|_\Phi$ is the $L^q$-norm and that $\eta_\Phi(s)=e^{s/q}$. By assumption, $\mu$ has a finite $\chi$-moment where $\chi(s) = e^{ps}$ for some $p > 0$. Since $q$ is large, we have $\eta_\Phi(4s)\lesssim \chi(s)$. By Theorem \ref{t:general-sp}, $\nu$ has a H\"older continuous super-potential with respect to $W^{1,2}$ and the norm $L^q$.

Let $\varphi \in W^{1,2}$ such that $\|\varphi\|_{W^{1,2}} \leq 1$. For $N\geq 1$, define $\varphi_N:=\min(|\varphi|,N)$. Then $\varphi_N$ belongs to a bounded subset of $W^{1,2}$ (cf. \cite[Prop. 4.1]{dinh-sibony:decay-correlations}). Define also $\psi_N:= \varphi_{N+1} - \varphi_N$. Notice that $0 \leq \psi_N \leq 1$, $\psi_N \equiv 0$ on $\{|\varphi| \leq N\}$, and $\psi_N \equiv 1$ on $\{|\varphi| \geq N + 1\}$. Therefore
$$\nu\{N\leq |\varphi|\leq N+1\} \leq \nu(\psi_{N-1}) \lesssim \|\psi_{N-1}\|_{L^q}^\beta\lesssim \area\{|\varphi|\geq N-1\}^{\beta/q},$$
where $\beta > 0$ is the H\"older exponent of the functional defined by $\nu$ and the area is with respect to $\omegaFS$.
 
From Proposition \ref{prop:exponential-estimate} it follows that  
$\area\{|\varphi|\geq N-1\} \lesssim e^{-\alpha' N^2}$
for some $\alpha'>0$, so
\begin{equation} \label{eq:nu-annuli}
\nu\{N\leq |\varphi|\leq N+1\}  \lesssim e^{-\alpha'' N^2} \text{ for some } \alpha'' >0.
\end{equation}

Now,  for $\theta > 0$  small enough, the first estimate in the corollary follows after cutting the integral $\int_{\P^1} e^{\theta|\varphi|^2}\,\diff \nu$ along the subsets $\{N\leq |\varphi|\leq N+1\}$ and using (\ref{eq:nu-annuli}).

It is not difficult to see that for our choice of $\Phi$ we have $\mathcal V_\epsilon(r)\lesssim r^\gamma$ for every $0< \gamma< 2 /q$. Then, the second estimate in the corollary follows by applying Proposition \ref{p:size-of-disc}.
\endproof

\begin{remark} \label{rmk:holder-regular}
A measure $m$ satisfying $m(\DD(a,r)) \leq c r^\alpha$ for some constants $c,\alpha >0$ is often called \textit{H\"older regular}. The H\"older regularity of $\nu$ under an exponential moment condition is an old result due to Guivarc'h and Raugi, see \cite[VI.4]{bougerol-lacroix}.
\end{remark}

\begin{corollary} \label{cor:reg-firstmoment}
Let $\mu$ be a non-elementary measure on $\PSL_2(\C)$ with finite first moment and let $\nu$ be the associated stationary measure. Then there are constants $c>0$ and $\alpha>0$ such that 
$$\nu(\DD(a,r))\leq c|\log r|^{-\alpha}$$
for every $a\in \P^1$ and $0<r\leq 1$.
\end{corollary}
\proof
As above, we will apply Theorem  \ref{t:general-sp} for a suitable function $\Phi$. Observe that the function 
$t\mapsto e^{-t^{-3}}$ is convex and increasing on some interval $[0,t_0]$ in $\R_{\geq 0}$. We extend it to a convex increasing function $\Phi$ on $\R_{\geq 0}$ such that $\Phi(t)=e^{t^2}$ for $t$ large enough.

\medskip\noindent
{\bf Claim 1.} We have $\eta_\Phi(4s)\lesssim s +1 = \chi(s) + 1$ for $s\geq 0$. 

\medskip

It is enough to prove that $\eta_\Phi(4s) \leq ks$ for some constant $k >0$ and $s$ large enough. Let $\varphi \in W^{1,2}$ be such that  $\int_{\P^1} \Phi(|\varphi|)\omegaFS = 1$.  We need to show that  
\begin{equation} \label{eq:eta-estimate}
\int_{\P^1} \Phi \Big( \frac{|\varphi|}{k' s} \Big)\omegaFS \leq  e^{-4s}
\end{equation}
for some constant $k'> 0 $ and $s$ large enough.

Observe that $\Phi(t)\lesssim e^{t^2}$ on $\R_{\geq 0}$. We have 
$$\int_{|\varphi|>2\sqrt{s}} \Phi \Big( \frac{|\varphi|}{s} \Big) \omegaFS \lesssim \int_{|\varphi|>2\sqrt{s}} e^{|\varphi|^2/s^2}\omegaFS \le e^{-s} \int_{|\varphi|>2\sqrt{s}} e^{|\varphi|^2}\omegaFS \leq e^{-s}\int_{\P^1} \Phi(|\varphi|)\omegaFS = e^{-s}.$$
On the other hand, we have
$$\int_{|\varphi|\leq 2\sqrt{s}} \Phi\Big(\frac{|\varphi|}{s} \Big)\omegaFS = \int_{|\varphi|\le 2\sqrt{s}} e^{-(|\varphi|/s)^{-3}}\omegaFS\le \int_{\P^1} e^{-s}\omegaFS =e^{-s}.$$ 
This gives (\ref{eq:eta-estimate}) for $k'=5$,  ending the proof of the claim.

\medskip

By Theorem  \ref{t:general-sp} and the claim, $\nu$ has a H\"older continuous super-potential with respect to $W^{1,2}$ and the distance $\dist_\Phi$.

To finish the proof, we now need to estimate the function $\mathcal{V}_\epsilon(r)$ appearing in Proposition \ref{p:size-of-disc}. It is enough to consider a fixed value of $\epsilon$. Take $\epsilon:=1/4$ and set $u:=u_{a,\epsilon}$.

\medskip\noindent
{\bf Claim 2.} We have $\|u\|_\Phi \leq |\log r|^{-1/8}$ for $r$ small enough.

\medskip
Set $A:=|\log r|^{-1/8}$. By the definition of $\|\cdot\|_\Phi$, we need to check that 
$$\int_{\P^1} \Phi \Big(\frac{|u|}{A} \Big)\omegaFS <1.$$

In order to simplify the notation, assume that $a=0$ and denote by $|z|$ the distance between $z$ and $0$. Then $u = |\log (|z| / 2r)|^{1/4}$ on $|z|< 2r$ and zero elsewhere. Observe that $|u|>A$ if and only if $|z|<2r e^{-|\log r|^{-1/2}}$. Moreover, we have  $|u|\leq |\log|z||^{1/4}$. Thus, using that $\omegaFS$ is comparable with $i \diff z \wedge \diff \overline z$ near $0$,
we have for $s:=-\log|z|$ and $r$ small
\begin{align*}
\int_{|u|>A} \Phi \Big( \frac{|u|}{A} \Big) \omegaFS &= \int_{|z|<2r e^{-|\log r|^{-1/2}}} \Phi \Big( \frac{|u|}{A} \Big) \omegaFS \lesssim \int_{|z|<3r} e^{|u|^2/A^2}\omegaFS \\ & \lesssim \int_{|\log r|-3}^\infty e^{A^{-2}s^{1/2}} e^{-2s} ds  \lesssim \int_{|\log r|-3}^\infty e^{-2s +2s^{3/4}} ds  \lesssim \int_{|\log r|-3}^\infty e^{-s} ds = O(r).
\end{align*}

Recall that $\Phi(0) = 0$ and that $u$ is supported by $\DD(a,2r)$. Then
$$\int_{|u|\leq A} \Phi \Big(\frac{|u|}{A} \Big) \omegaFS \lesssim \area (\DD(a,2r))=O(r^2).$$
The claim follows.

\medskip
The last claim gives that $\mathcal{V}_\epsilon(r) \lesssim |\log r|^{-\gamma}$ for $\epsilon = 1 / 4$, $r$ small and a suitable constant $\gamma >0$. The corollary then follows from Proposition \ref{p:size-of-disc}.
\endproof

\begin{remark} \label{rmk:BQregularity}
A similar type of regularity under a finite $p^{th}$ moment condition was obtained by Benoist-Quint in \cite{benoist-quint:CLT}. This is a crucial ingredient in their proof of the Central Limit Theorem. We note that the H\"older exponents appearing in this section can be made explicit. We chose not do so in order to keep the paper less technical. 
\end{remark}

\appendix

\section{Elementary sets and auxiliary lemmas} \label{sec:appendix}

We present in this appendix some results used in the text. A number of them are probably known to experts.

\vskip5pt

Let us first recall the classification of elements of $\Aut(\P^1)$. In this appendix we shall denote by the same symbol $g$ an element of $\Aut(\P^1)$, its corresponding matrix in $\SL_2(\C)$ and its class in $\PSL_2(\C)$. This should not cause any confusion.

 Recall that an element $g$ of $\Aut(\P^1)$ different from the identity is conjugated to either $z \mapsto z + 1$ or $z \mapsto \lambda z$ for some $\lambda \in \C \setminus \{0,1\}$. In  the former case, $g$ is called \textit{parabolic} and in the latter, $g$ is called \textit{elliptic} if $|\lambda| = 1$ or \textit{loxodromic} if $|\lambda| \neq 1$. A parabolic automorphism has a single fixed point that attracts every point of $\P^1$. An elliptic automorphism has two different neutral fixed points and a loxodromic automorphism $g$ admits two fixed points $a$ and $b$ such that $g^n(z) \to a$ and $g^{-n}(z) \to b$ as $n$ tends to infinity, for any $z \in \P^1 \setminus \{a,b\} $. In terms of the trace of the corresponding matrices, $g \neq \text{Id}$ is parabolic if $\Tr^2 g = 4$, elliptic if $\Tr^2 g \in [0,4)$ and loxodromic if $\Tr^2 g \notin [0,4]$.

\vskip5pt

Now let $R$ be a subset of $\PSL_2(\C)$. For $n\geq 1$ denote $$ R^n : = \{g_n\cdots g_1 : g_i \in R\} \quad \text{ and } \quad S^n:=\{gh^{-1} : g,h \in R^n\}.$$

Recall that $R$ is non-elementary if its support does not preserve a finite subset of $\P^1$ and if the semi-group generated by $R$ is not relatively compact, see Definition \ref{def:elementary} and Remark \ref{rmk:elementary}.

\begin{lemma}\label{Rn lox}
	Let $R$ be a non-elementary subset of $\PSL_2(\C)$. Then there exist integers $N_1 \geq 1$ and $N_2\geq1$ such that   $R^{N_1}$ contains  a loxodromic element and $S^{N_2}$ contains a non-elliptic element.
\end{lemma}

\begin{proof}
The first assertion is well known. First, we extend the action of $\PSL_2(\C)$ to the $3$-dimensional hyperbolic space $\mathbb H^3$ (see Remark \ref{rmk:elementary}-(iii)). Then, the results from \cite[Chapter 6]{das-tushar-urbanski} imply that the semi-group generated by $R$ contains a loxodromic element. This gives the first assertion.

We now prove the second assertion. Since $R^{N_1}$ is non-elementary, we can find another element $h_0$ in $R^{N_1}$ whose fix point set is different from that of $g_0$. If  $|\Fix(g_0)\cap\Fix(h_0)|=1$, Lemma \ref{1 common fixed point} below implies that $g_0h_0g_0^{-1}h_0^{-1}\in S^{2N_1}$ is parabolic. If $\Fix(g_0)\cap\Fix(h_0)=\varnothing$,  Lemmas \ref{elli-lox}, \ref{para-lox} and \ref{lox-lox} below show that there is an $N_3 \geq 1$ such that $g_0^{N_3} (h_0^{-1})^{N_3} \in S^{N_1 N_3}$ is loxodromic. This proves the second assertion and concludes the proof of the lemma.	
\end{proof}

\begin{lemma}\label{1 common fixed point}
If $g,h$ are two non-trivial elements in $\PSL_2(\C)$, $g$ has $2$ fixed points on $\P^1 $ and $|\Fix(g)\cap\Fix(h)|=1$, then $ghg^{-1}h^{-1}$ is parabolic.
\end{lemma}
\begin{proof}
See \cite[p.12]{maskit}.
\end{proof}

\begin{lemma}\label{elli-lox}
	Let $g,h \in \PSL_2(\C)$. If $g$ is loxodromic and $h$ is elliptic then there is an $N \geq 1$ such that $g^Nh^N$ is loxodromic.
\end{lemma}

\begin{proof}
	We can assume that the fixed points of $g$ are $0$ and $\infty$ and $g= \left(\begin{array}{cc} t&0\\0&t^{-1}\end{array} \right)$, where $|t|>1$. Since  $h$ is elliptic the set $\{h^n : n \geq 1 \}$ is relatively compact in $\PSL_2(\C)$. Hence, there exists a subsequence $h^{n_k}$, converging to some elliptic $r=\left(\begin{array}{cc}a&b\\c&d\end{array}\right) \in \PSL_2(\C)$. After replacing $h^{n_k}$ by $h^{2n_k}$ and $r$ by $r^2$ if necessary we may assume that $a\neq 0$. Denoting by $a_n,b_n,c_n,d_n$ the entries of $h^n$ we have that $a_{n_k} \to a$ and $d_{n_k} \to d$.  Then  $|\Tr^2(g^{n_k}h^{n_k})| = |t^{n_k} a_{n_k} + t^{-n_k} d_{n_k}|^2\to\infty$. If we choose $N$ so that $|\Tr^2(g^Nh^N)| > 4$ then $g^Nh^N$ is loxodromic.
\end{proof}

\begin{lemma}\label{para-lox}
	Let $g,h \in \PSL_2(\C)$. If $g$ is loxodromic and $h$ is parabolic then there exists an $N \geq 1$ such that  $g^Nh^N$ is loxodromic.
\end{lemma}
\begin{proof}
	We can write $g=A\left(\begin{array}{cc}t&0\\0&1/t\end{array}\right)A^{-1},h=B\left(\begin{array}{cc}1&1\\0&1\end{array}\right)B^{-1}$, where $|t|>1$ and $A,B\in\PSL_2(\C)$. Define $a_i,b_i,c_i,d_i$ by $A^{-1}B=\left(\begin{array}{cc}a_1&b_1\\c_1&d_1\end{array}\right)$ and  $B^{-1}A=\left(\begin{array}{cc}a_2&b_2\\c_2&d_2\end{array}\right)$. Then
	\begin{equation*} 
	\begin{split}
	\Tr(g^nh^n)&=\Tr\left(A\left(\begin{array}{cc}t^n&0\\0&1/t^n\end{array}\right)A^{-1}B\left(\begin{array}{cc}1&n\\0&1\end{array}\right)B^{-1}\right) \\
	&=\Tr\left(\left(\begin{array}{cc}t^n&0\\0&1/t^n\end{array}\right)A^{-1}B\left(\begin{array}{cc}1&n\\0&1\end{array}\right)B^{-1}A\right) \\
		&=\Tr\left(\left(\begin{array}{cc}t^n&0\\0&1/t^n\end{array}\right)\left(\begin{array}{cc}a_1&b_1\\c_1&d_1\end{array}\right)\left(\begin{array}{cc}1&n\\0&1\end{array}\right)\left(\begin{array}{cc}a_2&b_2\\c_2&d_2\end{array}\right)\right) \\
		&=a_1c_2nt^n+a_1a_2t^n+b_1c_2t^n+c_1b_2/t^n+c_1d_2n/t^n+d_1d_2/t^n\\
		&=a_1c_2nt^n+t^n+c_1d_2n/t^n+1/t^n,
	\end{split}
	\end{equation*}
which shows that  $|\Tr^2(g^nh^n)|$ is unbounded as $n\to \infty$.  Hence $g^n h^n$ is loxodromic for $n$ large enough.
\end{proof}

\begin{lemma}\label{lox-lox}
	If $g,h$ are both loxodromic and $\Fix(g)\cap\Fix(h)=\varnothing$, then $g^Nh^N$ is loxodromic for some $N \geq 1$.
\end{lemma}

\begin{proof}
		Write $g=A\left(\begin{array}{cc}t&0\\0&1/t\end{array}\right)A^{-1},h=B\left(\begin{array}{cc}s&0\\0&1/s\end{array}\right)B^{-1}$, where $|t|>1,|s|>1$ and $A,B\in\PSL_2(\C)$. We may assume that $|t| \geq |s|$. Define $a_i,b_i,c_i,d_i$ by $A^{-1}B=\left(\begin{array}{cc}a_1&b_1\\c_1&d_1\end{array}\right)$ and $B^{-1}A=\left(\begin{array}{cc}a_2&b_2\\c_2&d_2\end{array}\right)$. Then
		
		\begin{align*}
		\Tr(g^nh^n)&=\Tr\left(A\left(\begin{array}{cc}t^n&0\\0&1/t^n\end{array}\right)A^{-1}B\left(\begin{array}{cc}s^n&0\\0&1/s^n\end{array}\right)B^{-1}\right) \\
		&=\Tr\left(\left(\begin{array}{cc}t^n&0\\0&1/t^n\end{array}\right)A^{-1}B\left(\begin{array}{cc}s^n&0\\0&1/s^n\end{array}\right)B^{-1}A\right) \\
		&=\Tr\left(\left(\begin{array}{cc}t^n&0\\0&1/t^n\end{array}\right)\left(\begin{array}{cc}a_1&b_1\\c_1&d_1\end{array}\right)\left(\begin{array}{cc}s^n&0\\0&1/s^n\end{array}\right)\left(\begin{array}{cc}a_2&b_2\\c_2&d_2\end{array}\right)\right) \\
		&=a_1a_2t^ns^n+b_1c_2t^n/s^n+c_1b_2s^n/t^n+d_1d_2/t^ns^n. 
		\end{align*}
		
Suppose for contradiction that for every $n \geq1$, $g^n h^n$ is not loxodromic.  Then $|\Tr(g^nh^n)|$ is a bounded sequence. It follows that $a_1a_2=0$. Without loss of generality, assume $a_1=0$. We get $A^{-1}B= \left(\begin{array}{cc}0&b_1\\-1/b_1&d_1\end{array}\right)$ and $B^{-1}A = (A^{-1}B)^{-1}=\left(\begin{array}{cc}d_1&-b_1\\1/b_1&0\end{array}\right)$, so \begin{align*}
	h &= B\left(\begin{array}{cc}s&0\\0&1/s\end{array}\right)B^{-1} = A\left(\begin{array}{cc}0&b_1\\-1/b_1&d_1\end{array}\right)\left(\begin{array}{cc}s&0\\0&1/s\end{array}\right)\left(\begin{array}{cc}d_1&-b_1\\1/b_1&0\end{array}\right)A^{-1} \\ &= A\left(\begin{array}{cc} \ast &0\\ \ast & \ast\end{array}\right)A^{-1}.    
   \end{align*}	     
This implies that $h$ and $g$ have the same fixed point $A([0:1])$, contradicting the hypothesis  $\Fix(g)\cap\Fix(h)=\varnothing$. This proves the lemma.
	 	\end{proof}

\vskip5pt

For $g \in \PSL_2(\C)$ let $\theta_g(x) : = \log \frac{\| g \cdot v \|}{\| v\|}$, $x=[v]$. Then $\theta_g$ is a smooth function on $\P^1$ and $\|\theta_g\|_\infty = \log \|g\|$. The following estimate was used in Section \ref{sec:CLT}. 

\begin{lemma} \label{lemma:theta_g-estimate}
We have $\|\theta_g\|_{W^{1,2}} \lesssim 1 + \log \|g\|$.
\end{lemma}

\begin{proof}
Since $\|\theta_g\|_\infty = \log \|g\|$ it follows that $\| \theta_g \|_{L^1} \leq \log \|g\|$. So, from Proposition \ref{prop:equinorm} we only need to estimate $\|\del \theta_g\|_{L^2}$.

Set $\omega_g := i \del \theta_g \wedge \overline{\del \theta_g}$ so that $\|\del \theta_g\|^2_{L^2} = \int_{\P^1} \omega_g$. By Cartan's decomposition we can write $g = k' a k$ where $k,k' \in \text{SU}(2)$ and $a \in \SL_2(\C)$ is diagonal with positive eigenvalues. Since $k'$ and $k$ preserve the euclidean norm we have $$\theta_g(x) =  \log \frac{\| k'a k \cdot v \|}{\| v\|} =  \log \frac{\| a k \cdot v \|}{\| v\|} = \log \frac{\| a k \cdot v \|}{\|k \cdot v\|} = \theta_a(k \cdot x), $$
that is $\theta_g = k^* \theta_a$. Hence $\omega_g = k^* \omega_a$ and since $\text{SU}(2)$ is compact we have that $\omega_g \sim \omega_a$. This allows us to assume that $g$ is of the form $\left(\begin{smallmatrix} \lambda & 0 \\ 0 & \lambda^{-1} \end{smallmatrix}\right)$, for some $\lambda \geq 1$.

Let $z = [z:1]$ be the standard affine coordinate in $\P^1 \setminus \{\infty\}$. In this coordinate we have $g(z) = \lambda^2 z$, so   $$\theta_g(z) = \frac{1}{2}\log \frac{\lambda^4 |z|^2 + 1}{|z|^2 + 1} = \frac{1}{2}\log (\lambda^4 |z|^2 + 1) - \frac{1}{2}\log( |z|^2 + 1).$$

Hence $$\omega_g=i\del \theta_g \wedge\overline{\partial \theta_g}=\frac{(\lambda^4-1)^2|z|^2}{4(\lambda^4|z|^2+1)^2(|z|^2+1)^2}idz\wedge d\overline z.$$

Then, by Lemma \ref{lemma:omega-g-estimate} below we get $ \int_{\P^1} \omega_g \lesssim \log \lambda^4 = 4 \log \|g\|$. Hence $\|\del \theta_g\|_{L^2} \lesssim (\log\|g\|)^{1/2}$. This, together with the above estimate for $\|\theta_g\|_{L^1}$, implies the lemma.
\end{proof}

\begin{lemma}\label{lemma:omega-g-estimate}
Let $\beta > 1$ and denote by $z$ the standard affine coordinate in $\C \subset \P^1$. Then
$$\int_\C\frac{(\beta-1)^2|z|^2}{(\beta|z|^2+1)^2(|z|^2+1)^2}idz\wedge d\overline z \leq 2 \pi\frac{\beta-1}{\beta+1}\log\beta.$$
\end{lemma}

\begin{proof}
Multiplying the integral on  left hand side by $\frac{\beta+1}{\beta-1}$ gives
	\begin{align*}
	&\int\frac{(\beta^2-1)|z|^2}{(\beta|z|^2+1)^2(|z|^2+1)^2}idz\wedge d\overline z\leq \int\frac{(\beta^2-1)|z|^2}{(\beta^2|z|^4+1)(|z|^4+1)}idz\wedge d\overline z \\
	&=\iint \frac{(\beta^2-1)r^2}{(\beta^2 r^4+1)(r^4+1)} 2 r\,d r\, d \theta  =\iint \frac{\beta^2-1}{2(\beta^2 t+1)(t+1)}\,d t \,d\theta \\ &=  \pi \int \Big( \frac{\beta^2}{\beta^2t+1}-\frac{1}{t+1} \Big) \,d t = \pi \left[\log\frac{\beta^2t+1}{t+1}\right]_{0}^{\infty} = 2\pi \log\beta ,
	\end{align*}
giving the desired inequality.	
\end{proof}




\end{document}